\documentclass[12pt]{amsart}

\usepackage{amssymb}
\usepackage{amsmath}
\usepackage{amsthm}
\usepackage{enumerate}
\usepackage{graphicx}
\usepackage{blkarray}
\usepackage{mathrsfs}
\usepackage{float}

\usepackage{caption}
\usepackage{subcaption}

\usepackage{pgfplots}
\pgfplotsset{compat=1.15}
\usepackage{mathrsfs}
\usetikzlibrary{arrows}
\usetikzlibrary[patterns]

\usepackage[hidelinks]{hyperref}  
 
\numberwithin{equation}{section}

\usepackage{amsthm}

\makeatletter
\newtheorem*{rep@theorem}{\rep@title}
\newcommand{\newrepthm}[2]{%
\newenvironment{rep#1}[1]{%
 \def\rep@title{#2 \ref{##1}}%
 \begin{rep@theorem}}%
 {\end{rep@theorem}}}
\makeatother

\theoremstyle{plain}
\newtheorem{thm}{Theorem}[section]
\newtheorem{prop}[thm]{Proposition}
\newtheorem{lem}[thm]{Lemma}
\newtheorem{cor}[thm]{Corollary}

\newrepthm{thm}{Theorem}
\newrepthm{lem}{Lemma}
\newrepthm{prop}{Proposition}

\theoremstyle{definition}
\newtheorem{defn}[thm]{Definition}

\theoremstyle{remark}
\newtheorem{rem}[thm]{Remark}
\newtheorem*{acknowledgements}{Acknowledgements}

\theoremstyle{plain}

\newcommand{\thmref}[1]{Theorem~\ref{#1}}
\newcommand{\propref}[1]{Proposition~\ref{#1}}
\newcommand{\secref}[1]{Section~\ref{#1}}

\newcommand{\lemref}[1]{Lemma~\ref{#1}}
\newcommand{\corref}[1]{Corollary~\ref{#1}}
\newcommand{\figref}[1]{Figure~\ref{#1}}

\newcommand{\eqnref}[1]{Equation~\eqref{#1}}

\newcommand{\calT}{{\mathcal T}}

\newcommand{\CC}{{\mathbb C}}

\newcommand{\RR}{{\mathbb R}}

\newcommand{\ZZ}{{\mathbb Z}}

\newcommand{\what}{\widehat}
\newcommand{\wbar}{\overline}

\newcommand{\CHAT}{{\what{\CC}}}

\newcommand{\eps}{{\varepsilon}}

\DeclareMathOperator{\im}{Im}
\DeclareMathOperator{\re}{Re}
\DeclareMathOperator{\arcsinh}{arcsinh}

\DeclareMathOperator{\el}{EL}
\DeclareMathOperator{\area}{area}
\DeclareMathOperator{\length}{length}
\DeclareMathOperator{\conf}{conf}

\title[The extremal length systole of the cube]{The extremal length systole of the cube punctured at its vertices}
\author{Samuel Dobchies}
\address{
Département de mathématiques, Université de Sherbrooke, 2500, boulevard de l'Université, Sherbrooke (QC), J1K 2R1, Canada}
\email{samuel.dobchies@usherbrooke.ca}
\author{Maxime Fortier Bourque}
\address{
D\'epartement de math\'ematiques et de statistique, Universit\'e de Montr\'eal, 2920, chemin de la Tour, Montr\'eal (QC), H3T 1J4, Canada}
\email{maxime.fortier.bourque@umontreal.ca}

\begin{document}

\begin{abstract}
    We prove that the extremal length systole of the cube punctured at its vertices is realized by the 12 curves surrounding its edges and give a characterization of the corresponding quadratic differentials, allowing us to estimate its value to high precision. The proof uses a mixture of exact calculations done using branched covers and elliptic integrals, together with estimates obtained using either the geometry of geodesic trajectories on the cube or explicit conformal maps.
\end{abstract}


\maketitle

\section{Introduction}

Given a notion of length for closed curves on a surface, the correspon\-ding \emph{systole} is defined as the infimum of lengths of essential closed curves (those that are not homotopic to a point or a puncture) \cite{K07:SystolicGeometry}. Usually, length is calculated with respect to a Riemannian metric on the surface, such as the unique hyperbolic metric in a given conformal class. A different notion of length (defined purely in terms of the conformal structure) is also widely used in Teichm\"uller theory, namely, \emph{extremal length} (see \secref{sec:background}). The resulting \emph{extremal length systole} is a generalized systole in the sense of Bavard \cite{Bav:97} and has been studied in \cite{MGVP}, \cite{BolzaEL}, and \cite{NieWuXue}. So far, its exact value has been calculated for the $3$-times-punctured sphere, all $4$-times-punctured spheres, all tori, the regular octahedron punctured at its vertices, and the Bolza surface in genus $2$ \cite{BolzaEL}.

In the present paper, we compute the extremal length systole of the cube punctured at its vertices. We do not obtain an exact expression for its value, but we find a characterization which can in principle be used to estimate it to any desired precision. Our criterion can be stated as follows.

\begin{thm} \label{thm:main}
For $r<0$ and $z\in \CC$, let
\[
F_r(z) = \frac{z-r}{z(z^2-10z+1)(3z^2+2z+3)}.
\]
Then the extremal length systole of the cube punctured at its vertices is realized by the $12$ curves surrounding its edges and is equal to
\[
\inf_{r<0} \frac{4\int_0^{5 - 2\sqrt{6}} \sqrt{|F_{r}(z)|}dz }{ \min\left( \int_{r}^0 \sqrt{|F_{r}(z)|}dz , \int_{5 - 2\sqrt{6}}^{5 + 2\sqrt{6}} \sqrt{|F_{r}(z)|}dz \right)}.
\]
\end{thm}

We also express the extremal length systole as the supremum of a slightly more complicated expression involving the integral of $\sqrt{|F_{r}|}$ over various intervals in \propref{prop:formulas}. By evaluating the lower and upper bounds at \[r = -9.0792887\](which is close to the unique point $r_0$ where they agree) with interval arithmetic, we obtain that the extremal length systole is contained in the interval \[ (2.84317,2.843187).\] 

While cases where the extremal length systole is given by a closed formula (as in the examples computed in \cite{BolzaEL}) are probably rare, the above result indicates that it might still be effectively computable in general. On the other hand, we rely heavily on the symmetries of the cube to prove \thmref{thm:main}. On a general punctured sphere, the calculation of the extremal length systole could be reduced to a finite collection of finite-dimensional optimization problems, but not necessarily to a single $1$-dimensional problem as for the cube.

\subsection*{The general strategy}

The general strategy for calculating the extremal length systole of a Riemann surface, developed in \cite{BolzaEL}, goes as follows:
\begin{enumerate}
    \item guess which essential curves have the smallest extremal length;
    \item find an upper bound $U$ for the extremal length of these curves;
    \item find a lower bound $L> U$ for the extremal length of all competing curves.
\end{enumerate}

Although the definition of extremal length makes it easy to bound from below (any conformal metric yields a lower bound), the estimates required for step (3) to succeed for curves that are not too complicated are usu\-al\-ly quite delicate. In \cite{BolzaEL}, the authors were able to compute the extremal length of the first and second shortest curves on the punctured octahedron explicitly in terms of elliptic integrals. Together with coarser estimates for the remaining curves, this allowed them to determine the extremal length systole. Two important points to emphasize are that lower bounds coming from natural test metrics were not good enough to show that the second shortest curves are longer than the shortest ones, and that exact calculations are rare. Indeed, we only know how to compute extremal length exactly for highly symmetric curves on highly symmetric surfaces. More precisely, we know how to do this if the quotient of the surface by the group of automorphisms that preserve the homotopy class of the curve has a deformation space of real dimension at most $2$.

\subsection*{Outline of the proof}

In the case of the cube, we are able to obtain explicit formulas in \secref{sec:explicit} for the extremal length of the curves that surround the faces (\figref{fig:face}) or the face diagonals (\figref{fig:diagonal}), the hexa\-go\-nal curves obtained by intersecting the cube with a plane that bisects a main diagonal perpendicularly (\figref{fig:hexagon}), as well as two other types of curves (Figures \ref{fig:staple} and \ref{fig:tripod}), in terms of elliptic integrals. To do these calculations, we find branched covers from the cube to pillowcases (doubles of rectangles). In two of these cases, we need to use the following interesting geometric fact.

\begin{replem}{lem:triangle}
If a cube is inscribed in the unit sphere in such a way that two of its vertices lie at $(0,0,\pm1)$, then the stereographic projection sends the remaining six vertices to the vertices and the midpoints of the sides of an equilateral triangle in the plane.
\end{replem}

\begin{figure}
    \centering
    \subcaptionbox{face\label{fig:face}}{\includegraphics[width=0.3\textwidth]{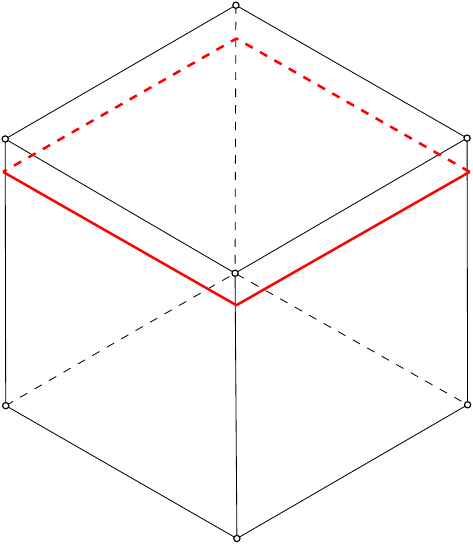}}
    \hfill
    \subcaptionbox{diagonal\label{fig:diagonal}}{\includegraphics[width=0.3\textwidth]{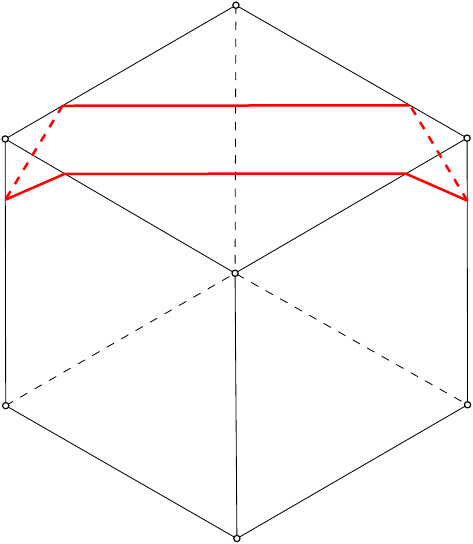}}
    \hfill
    \subcaptionbox{staple\label{fig:staple}}{\includegraphics[width=0.3\textwidth]{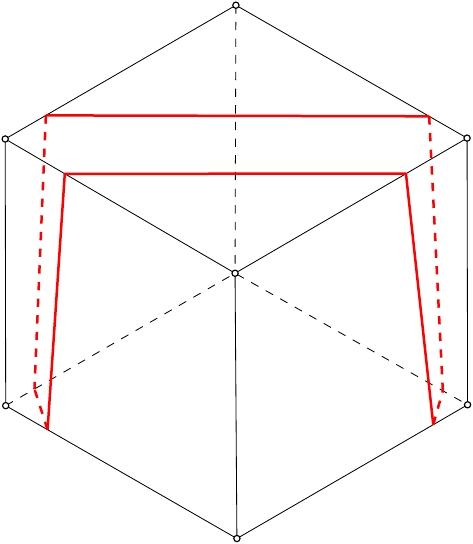}}

    \vspace{.3cm}

    \subcaptionbox{hexagon\label{fig:hexagon}}{\includegraphics[width=0.3\textwidth]{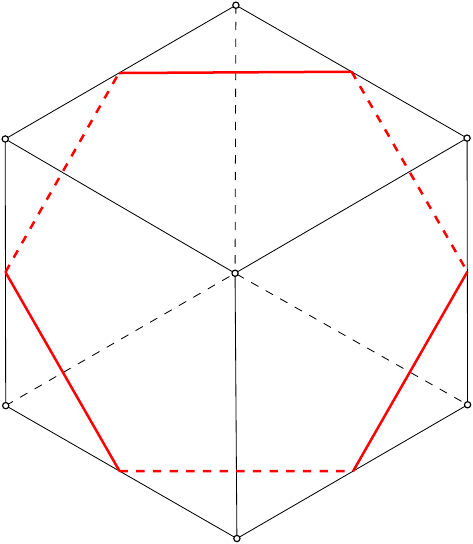}}
    \hfill
    \subcaptionbox{tripod\label{fig:tripod}}{\includegraphics[width=0.3\textwidth]{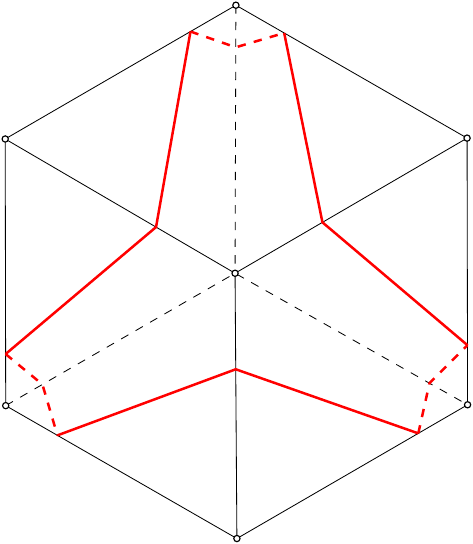}}
    \hfill
    \subcaptionbox{edge\label{fig:edge}}{\includegraphics[width=0.3\textwidth]{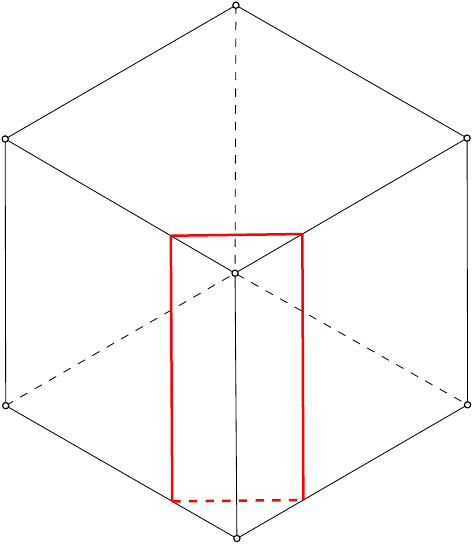}}
    
    \vspace{.3cm}
    
    \subcaptionbox{triangle\label{fig:triangle}}{\includegraphics[width=0.3\textwidth]{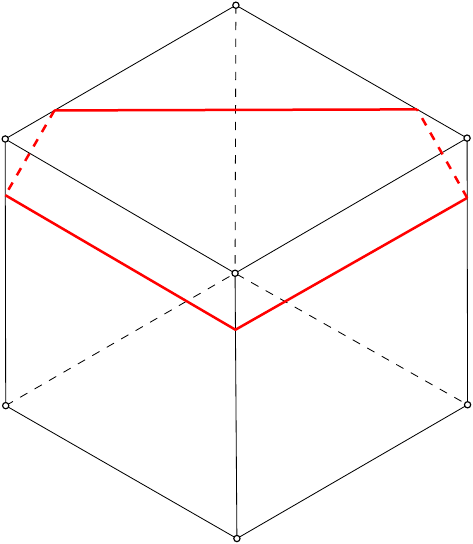}}
    \hfill
    \subcaptionbox{double-edge\label{fig:double-edge}}{\includegraphics[width=0.3\textwidth]{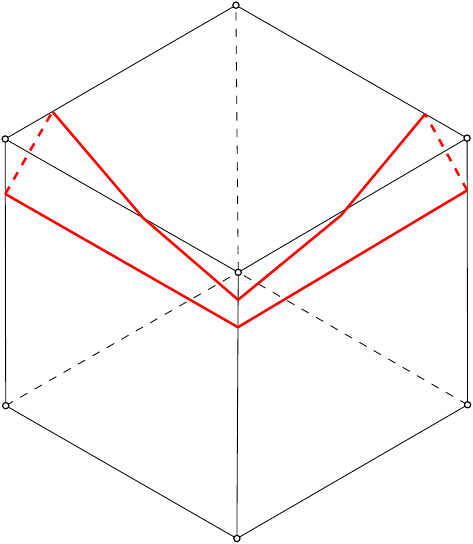}}
    \hfill
    \subcaptionbox{diamond\label{fig:diamond}}{\includegraphics[width=0.3\textwidth]{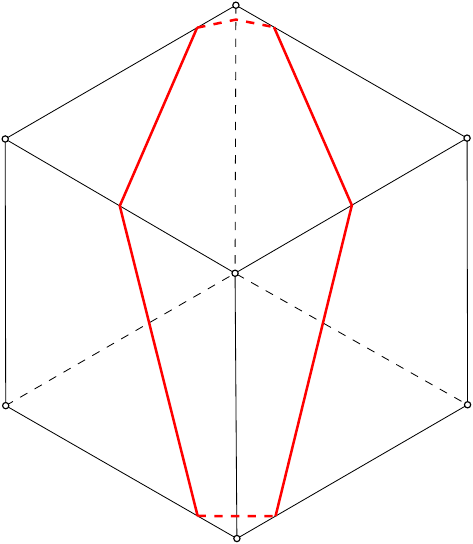}}
    
    \caption{A zoo of essential simple closed curves on the punctured cube.}
    \label{fig:three graphs}
\end{figure}

As stated in \thmref{thm:main}, the curves with the smallest extremal length are the simple closed curves that surround edges of the cube. These curves admit a dihedral symmetry of order $4$, but the quotient by this symmetry yields an arc on a disk with four marked points on the boundary and one in the interior. The deformation space of such surfaces is $3$-dimensional and so is the space of integrable holomorphic quadratic differentials that are real along the boundary on any such punctured quadrilateral. To compute the required extremal length in \secref{sec:edge}, we extract a $1$-parameter family of Jenkins--Strebel differentials with at most two rectangles that contains the one we are looking for, and figure out a way to detect when one of the rectangles is absent. This characterizes the desired quadratic differential and yields the formula given in \thmref{thm:main}. Most of these arguments apply to any proper arc on a once-punctured polygon, hence could be used in future work. In \secref{sec:edge}, the results are phrased in terms of the six-times-punctured sphere obtained by doubling the punctured quadrilateral across its boundary, but the two formulations are equivalent.

For most curves besides the above ones, we can use the Euclidean metric on the cube to get good enough lower bounds on their extremal length. This requires estimating lengths of geodesic trajectories on the cube, a calculation which is simplified by the following result.

\begin{repthm}{thm:rotation}
For any vertex-to-vertex geodesic trajectory on the surface of the regular tetrahedron, cube, octahedron, or icosahedron, there is an order two rotation of the polyhedron that sends the trajectory to itself. In particular, there does not exist a geodesic trajectory from a vertex to itself on these Platonic solids.
 \end{repthm}

 The second part of this theorem was first proved in \cite{Davis} for the tetrahedron and the cube, and in \cite{Fuchs} for the octahedron and the icosahedron. A unified proof was given in \cite[Theorem 1.3]{AthreyaAulicinoHooper} and then in \cite{Troubetzkoy}. The first part (although not stated explicitly) was also proved in \cite{Troubetzkoy}. We give a slightly shorter proof in \secref{sec:trajectories}. Note that on the regular dodecahedron, there does exist a geodesic trajectory from a vertex to itself \cite{AthreyaAulicino} (see also \cite{AthreyaAulicinoHooper}). The difference stems from the fact that equilateral triangles and squares tile the plane while regular pentagons do not.

\thmref{thm:rotation} and some arguments in its proof allow us to compute the bottom of the simple length spectrum of the flat metric on the punctured cube.

\begin{repprop}{prop:simple_length_spectrum}
The four smallest entries in the simple length spectrum of the unit cube punctured at its vertices are $2$, $2\sqrt{2}$, $2 + \sqrt{2}$, $4$, and $3\sqrt{2}$. These are realized only by the edge curves, the diagonal curves, the triangle curves, the face curves or double-edge curves, and the hexagon curves, respectively.
\end{repprop}

We then use this to get lower bounds on extremal lengths. Unfortunately, the resulting bound is not good enough to rule out triangle curves (\figref{fig:triangle}) and double-edge curves (\figref{fig:double-edge}) as extremal length systoles. We thus need finer estimates, which we obtain using Minsky's inequality in both cases together with lower bounds on the extremal length of other curves that intersect them $4$ times. For triangle curves, we need a lower bound on the extremal length of the diamond curve pictured in \figref{fig:diamond}. We obtain such a lower bound by embedding a cylinder in the homotopy class of a diamond curve through explicit conformal maps. 

\thmref{thm:main} is then obtained in \secref{sec:last} by combining the various lower bounds on extremal lengths of curves together with our numerical upper bound for the edge curves.

\begin{acknowledgements}
Samuel Dobchies was supported by two Undergraduate Student Research Awards (in 2022 and 2023) from the Natural Sciences and Engineering Re\-sear\-ch Council of Canada (NSERC). Maxime Fortier Bourque was supported by Discovery Grant RGPIN2022-03649 from NSERC. We thank the referees for their careful reading and useful comments.
\end{acknowledgements}

\section{Background} \label{sec:background}

In this section, we recall the definition and some properties of extremal length. We refer the reader to \cite{BolzaEL} and \cite{AhlforsConf} for more details and further references.

\subsection*{Extremal length}

Any Riemann surface $Y$ comes equipped with a conformal class $\conf(Y)$ of semi-Riemannian metrics, namely, those that push-forward to non-negative multiples (by a Borel-measurable function) of the Euclidean metric under conformal charts from $Y$ to the plane. By convention, we also require the metrics in $\conf(Y)$ to have finite positive area.

Given a closed curve $\alpha$ in $Y$, its \emph{extremal length} is defined as
\[
\el(\alpha, Y) := \sup_{h \in \conf(Y)} \frac{\ell([\alpha],h)^2}{\area(Y,h)},
\]
where $[\alpha]$ is the homotopy class of $\alpha$ and for a set $c$ of curves,
\[
\ell(c,h):= \inf \{ \length(\gamma,h) : \gamma \in c\}
\]
is the infimum of the lengths of curves in $c$ with respect to $h$. It is obvious from the definition that extremal length is invariant under conformal or anti-conformal diffeomorphisms.

\subsection*{Extremal length systole} It is well-known that $\el(\alpha, Y)=0$ if and only if $\alpha$ is \emph{inessential}, meaning that it can be homotoped into an arbitrarily small neighborhood of a point or a puncture in $Y$. A closed curve that is not inessential is called \emph{essential}. The \emph{extremal length systole} of $Y$ is then defined as the infimum of $\el(\alpha, Y)$ over all essential closed curves $\alpha$ in $Y$. It was shown in \cite[Theorem 3.2]{BolzaEL} that if $Y$ has a finitely generated fundamental group and is different from the thrice-punctured sphere, then the infimum is only realized by \emph{simple} (i.e., embedded) closed curves.

\subsection*{Quadratic differentials}

A \emph{quadratic differential} on a Riemann surface $Y$ is a map $q: TY \to \CC$ such that $q(\lambda v) = \lambda^2 q(v)$ for every $\lambda \in \CC$ and every $v \in TY$. In any chart $\phi$, we can write $q = f_\phi(z)dz^2$ for some function $f_\phi$. A quadratic differential is \emph{holomorphic} (resp. \emph{meromorphic}) if the local functions $f_{\phi}$ are holomorphic (resp. meromorphic). It is \emph{integrable} if $\int_Y |q| < \infty$. It is known that a meromorphic quadratic differential on a closed Riemann surface is integrable if and only if all its poles are simple \cite[p.24]{Strebel}.

A \emph{trajectory} of a meromorphic quadratic differential $q$ is a maximal smooth curve $\gamma$ (closed or not) in $Y$ such that the argument of $q(\gamma'(t))$ is constant. A standard uniqueness result for ordinary differential equations implies that trajectories are always simple. A trajectory is \emph{horizontal} if $q(\gamma'(t))>0$ for all $t$. A trajectory is \emph{critical} if it tends to a zero, a pole, or a puncture in at least one direction, and is \emph{regular} otherwise.

From a zero of order $k\geq -1$ (interpreted as a simple pole if $k=-1$) of a meromorphic quadratic differential emanate $k+2$ half-trajectories \cite[Section 7.1]{Strebel}, called \emph{prongs}, some pairs of which may belong to the same trajectory.

The \emph{Euler--Poincaré formula} \cite[Proposition 5.1]{FLP} states that a meromorphic quadratic differential on a closed Riemann surface $Y$ of genus $g$ has $-2\chi(Y)=4g-4$ zeros counting multiplicity, where a pole is counted as a zero of multiplicity $-1$. There is also a version for surfaces with boundary that can be obtained via doubling. The order of a zero on the boundary is counted with weight $1/2$ in that version.

\subsection*{The theorem of Jenkins} 
 
For an essential simple closed curve $\alpha$ on a Riemann surface $Y$ with a finitely generated fundamental group, a theorem of Jenkins \cite{Jenkins} states that there exists an integrable holomorphic quadratic differential $q_\alpha$ on $Y$, unique up to positive scaling, whose regular ho\-ri\-zon\-tal trajectories are all homotopic to $\alpha$. If $Y$ has a non-empty ideal boundary, then $q_\alpha$ should extend to be analytic and non-negative along it.

Geometrically, this means that $Y$ can be obtained from a Euclidean cylinder by partitioning its ends into subintervals and gluing a subset of them in pairs in such a way that the curves going around the cylinder (the \emph{core} curves) become homotopic to $\alpha$ in $Y$. More generally, a quadratic differential is \emph{Jenkins--Strebel} if the complement of its critical horizontal trajectories is a union of cylinders foliated by closed horizontal trajectories.

The second part of Jenkins's theorem states that the supremum in the de\-fi\-nition of $\el(\alpha, Y)$ is attained if and only if the metric $\rho$ is equal almost everywhere to a constant multiple of the conformal metric with area form $|q_\alpha|$.

Finally, the third part of Jenkins's theorem asserts that \[\el(\alpha, Y)= \inf \el(A)\] where the infimum is over all embedded annuli $A \subseteq Y$ whose core curves are homotopic to $\alpha$ and $\el(A)$ denotes the extremal length of the set of core curves in $A$ (this is equal to the circumference divided by the height if $A$ is represented as a Euclidean cylinder). If $Y$ is not a torus, then equality is achieved only if $A$ is equal to the complement of the critical horizontal trajectories of $q_\alpha$. If $Y$ is a torus, then equality is achieved only if $A$ is the complement of a (regular) horizontal trajectory of $q_\alpha$.

\subsection*{Behaviour under branched covers}

One situation which makes it easier to identify the quadratic differential $q_\alpha$ is if $Y$ and $\alpha$ are both invariant under some finite group of conformal automorphisms $G$. In that case, the quotient map $f: Y \to Y/G$ is holomorphic and $q_\alpha$ is the pull-back by $f$ of the quadratic differential $q_\beta$ corresponding to $\beta= f(\alpha)$ on $Y/G$ punctured at the critical values of $f$.

More generally, extremal length behaves as follows under holomorphic maps \cite[Lemma 4.1]{BolzaEL}.

\begin{lem}\label{lem:el_under_covers}
Let $f : Z \to W$ be a holomorphic map of degree $d$ between Riemann surfaces
with finitely generated fundamental groups, let $Q \subset W$ be a finite set containing the critical
values of $f$, and let $P \subseteq f^{-1}(Q)$ be such that $f^{-1}(Q) \setminus P$ is a subset of the critical points of $f$. Then 
\[
\el(f^{-1}(\gamma), Z \setminus P) = d \cdot \el(\gamma, W \setminus Q)
\]
for any simple closed curve $\gamma$ in $W \setminus Q$.
\end{lem}

Although the statement does not mention quadratic differentials, the proof shows that the quadratic differential for $f^{-1}(\gamma)$ on $Z \setminus P$ is the pull-back of the one for $\gamma$ on $W\setminus Q$. Thus, if we know explicit single-cylinder quadratic differentials on the base surface $W\setminus Q$, we can lift them to more complicated surfaces. In practice, we instead start with a pair $(\alpha, Y)$ and look for a map $f$ such that $(\alpha, Y)=(f^{-1}(\gamma), Z \setminus P)$.

In the present paper, we will apply this lemma with $Z=W=\CHAT$ (the Riemann sphere), $f$ a rational map, $P$ some set of punctures such that $Y=\CHAT\setminus P$, and $Q$ equal to the union of $f(P)$ and the set $V$ of critical values of $f$. In our examples, we will have that $P=f^{-1}(f(P))$ and that $f^{-1}(V)$ is the set of critical points of $f$, so that the condition that $f^{-1}(Q) \setminus P$ is a subset of the critical points will be satisfied.

\subsection*{Elliptic integrals}

In the case that $Y$ is the sphere punctured at four points that lie on a circle and $\alpha$ is invariant under inversion in that circle, then we can compute $\el(\alpha,Y)$ explicitly in terms of certain elliptic integrals, as we will now explain. 

For $k\in (0,1)$, the \emph{complete elliptic integral of the first kind} at \emph{modulus} $k$ is given by
\[
K(k):= \int_0^1 \frac{dt}{\sqrt{(1-t^2)(1-k^2t^2)}}.
\]
The \emph{complementary modulus} is $k':= \sqrt{1-k^2}$ and the \emph{complementary integral} is
\[
K'(k):=K(k')= \int_1^{1/k}\frac{dt}{\sqrt{(1-t^2)(1-k^2t^2)}},
\]
where the second equality can be proved by a change of variable (see \cite[Lemma 5.2]{BolzaEL}). We will also write
\[
(K/K')(k):= K(k)/K'(k) = K(k)/K(k') = K(k)/K(\sqrt{1-k^2})
\]
as these ratios come up often.

The  following formula is proved implicitly in \cite[Theorem A.1]{BolzaEL}.

\begin{lem} \label{lem:formule4pt}
Let $Y=\CC\setminus\{0, 1, t\}$ for some $t > 1$ and let $\alpha$ be a simple closed curve in $Y$ separating $(0,1)$ from $t$. Then $\el(\alpha, Y) = 2 (K/ K')(1/\sqrt{t})$.
\end{lem}

\begin{proof}
 Let $k=1/\sqrt{t} \in (0,1)$. By Lemma 5.2 in \cite{BolzaEL}, the extremal length of the curve $\beta$ in $X = \CHAT \setminus \{ \pm 1, \pm 1/k \}$ separating the interval $(-1,1)$ from $\pm 1/k$ is equal to $4(K/K')(k)$. 

The degree $2$ holomorphic map $f(z)=z^2$ sends $\pm1$ to $1$ and $\pm 1/k$ to $1/k^2=t$, and has critical points $0$ and $\infty$, which are fixed. By \lemref{lem:el_under_covers}, we have
\[
\el(\alpha,Y) = \frac12 \el(\beta,X) = 2 (K/K')(k)
\]
since $f^{-1}(\alpha)$ is homotopic to $\beta$.
\end{proof}

\begin{rem}
Observe that $t$ is equal to the cross-ratio of $(0,1,t,\infty)$. Hence, for any $4$-tuple of points that lie on a circle in $\CHAT$, the above lemma gives a formula for the extremal length of a symmetric curve separating them into two pairs, in terms of their cross-ratio (calculated in the correct order). 
\end{rem}

Note that \cite[Lemma 5.2]{BolzaEL} cited above is proved by showing that
\[
q_\beta = \frac{dz^2}{(1-z^2)(1-k^2z^2)}
\]
and then expressing the circumference and area of the horizontal cylinder of $q_\beta$ in terms of elliptic integrals. One can also show directly  that
\[
q_\alpha = \frac{dz^2}{z(1-z)(t-z)}
\]
(see below), but expressing the extremal length in terms of the complete elliptic integral $K$ makes things more convenient for numerical calculations. Indeed, there are algorithms that compute $K$ using an iteration scheme rather than numerical integration and provide certified error bounds. This is based on the fact that
\[
K(k)= \frac{\pi}{2 M (1, k')}
\]
where $M(a,b)$ is the \emph{arithmetic-geometric mean} of $a$ and $b$ \cite[Theorem 1.1]{AGM}. The latter can be estimated precisely and rigorously using interval arithmetic. For instance, in the computer algebra system \texttt{SageMath} \cite{sagemath}, the command \texttt{CBF(b).agm1()} returns a true interval that contains $M(1,b)$. We use this to estimate various ratios of elliptic integrals in the next section.

We also record here the fact that the formula in \lemref{lem:formule4pt} still works if we add more punctures between $0$ and $1$ or between $t$ and $+\infty$.

\begin{lem} \label{lem:add_punctures}
Let $t>1$, let $Y=\CC\setminus\{0, 1, t\}$, let $\alpha$ be a simple closed curve in $Y$ separating $(0,1)$ from $t$, and let $Z = Y \setminus B$ where $B \subseteq (0,1)\cup (t,+\infty)$ is a finite union of closed intervals (possibly reduced to points). Then $\el(\alpha, Z)=\el(\alpha, Y)=2 (K/K')(1/\sqrt{t})$.
\end{lem}
\begin{proof}
The quadratic differential realizing the extremal length of $\alpha$ on $Y$ is equal to 
\[
q_\alpha = \frac{dz^2}{z(1-z)(t-z)}
\]
because its critical horizontal trajectories are $[0,1]$ and $[t,+\infty]$ and their complement is a topological cylinder with core curves homotopic to $\alpha$. After removing the set $B$, the restriction of $q_\alpha$ to $Z$ is still an integrable holomorphic quadratic differential on $Z$ whose non-critical horizontal trajectories are all homotopic to $\alpha$. Thus, by the theorem of Jenkins,  $q_\alpha$ realizes the extremal length of $\alpha$ on $Z$. Since neither the length of the closed trajectories nor the area has changed, the extremal length of $\alpha$ is unchanged as well.
\end{proof}

\section{Explicit calculations} \label{sec:explicit}

In this section, we explicitly compute the extremal length of various essential simple closed curves on the cube punctured at its vertices in terms of elliptic integrals. Before we start, we explain how we can work on the Riemann sphere rather than the surface of the cube.

\begin{lem} \label{lem:radial}
If a cube is inscribed in a sphere, then there is a conformal map from the cube minus its vertices to the sphere minus these vertices.
\end{lem}

\begin{proof}
The isometry group $G$ of the cube acts simply transitively on a tiling $\calT$ of its faces by isosceles right triangles (see \figref{fig:tiling}), as well as on the tiling obtained by projecting $\calT$ radially onto the circumscribed sphere. By the Riemann mapping theorem and the fact that the conformal automorphisms of the disk act transitively on triples of points appearing in counterclockwise order along the boundary, there exists a conformal homeomorphism between any one of these triangles on the cube and its radial projection on the sphere that sends vertices to vertices. We can then extend this conformal map from the cube to the sphere equivariantly with respect to $G$. This conformal map fixes the vertices of the cube, as these are isolated fixed points of some subgroups of $G$.
\end{proof}

\begin{figure}
\centering
\includegraphics[width=0.3\textwidth]{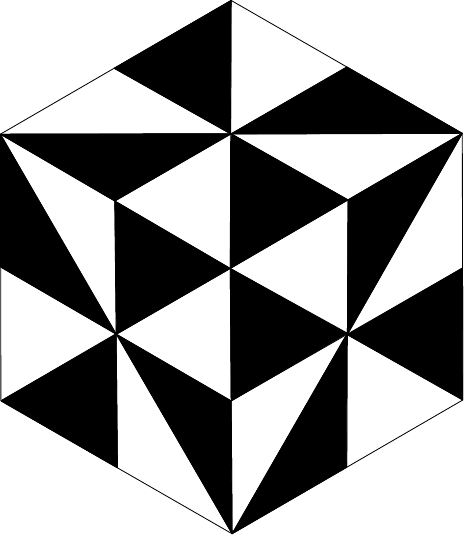}
\caption{A tiling of the surface of the cube on which its isometry group acts simply transitively.}
\label{fig:tiling}
\end{figure}

\begin{rem}
The same argument applies to any of the five Platonic solids.
\end{rem}

For various curves on the cube that are invariant under a rotation, our goal is to quotient out by that rotation to simplify the surface and then use Lemmas \ref{lem:el_under_covers}, \ref{lem:formule4pt}, and \ref{lem:add_punctures} to obtain a formula for the extremal length in terms of elliptic integrals. To find the quotient map $f$, we first need to compute the location of the eight vertices of the cube on the Riemann sphere and then apply a M\"obius transformation in order to put the fixed points of the rotation at $0$ and $\infty$. 

We will use the following terminology throughout the paper.

\begin{defn}
A simple closed curve $\gamma$ in a surface \emph{surrounds} a contractible set $E$ if for any open neighborhood $V$ of $E$, $\gamma$ can be homotoped into  $V \setminus E$. 
\end{defn}

\subsection{The face curves}

Let $X$ be the cube punctured at its vertices. We define a \emph{face curve} on $X$ to be a curve that surrounds a face (see \figref{fig:face}). Since any two face curves are related by an homotopy and an isometry of the cube, their extremal lengths on $X$ are equal. To compute this extremal length, we first need to compute the location of the vertices of the cube after stereographic projection.

\begin{lem} \label{lem:plus_sign}
The cube minus its vertices is conformally equivalent to \[\CHAT \setminus \{ \pm a, \pm i a, \pm 1/a, \pm i/a \},\] where $a = (\sqrt{3}+1)/\sqrt{2}$.
\end{lem}
\begin{proof}
Consider a cube inscribed in the unit sphere in $\RR^3$ in such a way that the top and bottom faces are parallel to the $xy$-plane while the vertices all lie in the other two coordinate planes (which cut the cube diagonally). By \lemref{lem:radial}, the cube minus its vertices is conformally equivalent to the sphere minus the same vertices. It only remains to compute the images of the vertices by the stereographic projection from the unit sphere to $\CHAT$.

Recall that the stereographic projection is given by
\[
S(x,y,z) = \frac{x}{1-z} + i\frac{y}{1-z}.
\]
One vertex of the cube is at $(\sqrt{2}/\sqrt{3},0,1/\sqrt{3})$ and we compute \[
S(\sqrt{2}/\sqrt{3},0,1/\sqrt{3})=\frac{\sqrt{2}}{\sqrt{3}-1} =\frac{\sqrt{3}+1}{\sqrt{2}}=a.\]
The other vertices can be obtained from this one by applying the rotation of angle $\pi/2$ around the $z$-axis, the reflection in the $xy$-plane, and their compositions. Under $S$, these transformations correspond to the rotation of angle $\pi/2$ around the origin and the inversion in the unit circle, yielding the points $\{ \pm a, \pm i a, \pm 1/a, \pm i/a \}$.
\end{proof}

We can now compute the extremal length of the face curves on $X$.

\begin{prop} \label{prop:face}
The extremal length of any face curve on the cube minus its vertices is equal to
\[
8\, (K / K')(1/a^4) \in [3.12680384539222 \pm 8.07\times 10^{-15}],
\]
where $a = (\sqrt{3}+1)/\sqrt{2}$.
\end{prop}
\begin{proof}
Position the cube as in the previous lemma. Then the face curve obtained by intersecting this cube with the $xy$-plane is sent to the unit circle under the conformal map in question. The extremal length is therefore equal to $\el(\alpha, \CHAT \setminus \{ \pm a, \pm i a, \pm 1/a, \pm i/a \})$ where $\alpha$ is the unit circle. 

We then apply the map $f(z) = (az)^4$ to quotient out the four-fold symmetry. This map sends four vertices to $1$ and the other four to $a^8$, and has critical values at $0$ and $\infty$ where we must puncture. We thus let $W = \CC \setminus \{0, 1, a^8 \}$. The map $f$ sends $\alpha$ to $\beta$ traced $4$ times, where $\beta$ is the circle of radius $a^4$, so that $f^{-1}(\beta)=\alpha$. By \lemref{lem:formule4pt}, we have
\[\el(\beta, W) = 2\, (K/K')(1/\sqrt{t})\]
where $t = a^8$ and by \lemref{lem:el_under_covers} we have
\[ \el(\alpha, Z) = 4 \, \el(\beta, W) = 8 \, (K/K')(1/\sqrt{t})  = 8 \, (K/K')(1/a^4).\]
The rigorous enclosure
\[ \el(\alpha, Z) \in [3.12680384539222 \pm 8.07\times 10^{-15}] \]
is obtained in \texttt{SageMath} (see the ancillary file \href{https://arxiv.org/src/2404.00336/anc}{\texttt{integrals}}, available at \href{https://arxiv.org/src/2404.00336/anc}{arxiv.org/src/2404.00336/anc}).
\end{proof}

\subsection{The diagonal curves}

We define a \emph{diagonal curve} on $X$ to be a curve that surrounds a diagonal of a face (see \figref{fig:diagonal}). We can use the same configuration as in the previous subsection to compute the extremal length of diagonal curves.

\begin{prop} \label{prop:diagonal}
The extremal length of any diagonal curve on the cube minus its vertices is equal to
\[
4\, (K / K')(1/\sqrt{t}) \in  [4.1335929781133 \pm  2.81\times 10^{-14}],
\]
where $t = 2a^4/(1+a^4)$ and $a = (1+\sqrt{3})/\sqrt{2}$.
\end{prop}

\begin{proof}
Position the cube as in the previous subsection. Under the conformal map of \lemref{lem:plus_sign}, one of the diagonal curves on the bottom of the cube is sent to the curve $\alpha$ surrounding $[-1/a,1/a]$ in the plane.

We then apply the squaring map $f(z) = z^2$, which sends the punctures to $-a^2, -1/a^2, 1/a^2$, and $a^2$, and has critical points and values at $0$ and $\infty$. The diagonal curve $\alpha$ is sent in a $2$-to-$1$ fashion onto a curve $\beta$ surrounding $[0,1/a^2]$. By \lemref{lem:el_under_covers}, the extremal length of $\alpha$ in \[\CHAT \setminus \{ \pm a, \pm i a, \pm 1/a, \pm i/a \}\] is twice the extremal length of $\beta$ in $\CHAT \setminus \{0, \pm 1/a^2, \pm a^2, \infty \}$ (recall that we need to puncture at the critical values as well as the images of the punctures).

We then scale by $a^2$ to obtain $\CHAT \setminus \{0, \pm 1, \pm a^4, \infty\}$ and apply the M\"obius transformation $g(z) = \frac{2z}{z+1}$, which fixes $0$ and $1$, and sends $-1$ to $\infty$. This leaves us with punctures at $g(-1)=\infty$, $g(0)=0$, $g(1)=1$, 
\[
t_1 = g(a^4) = \frac{2a^4}{1+a^4}, \quad t_2 = g(\infty) = 2, \quad \text{and} \quad t_3 = g(-a^4) = \frac{2a^4}{a^4 - 1}.
\]
This sends the curve $\beta$ to a curve $\gamma$ surrounding $[0,1]$. Since $a>1$, it is easy to check that $1<t_1<t_2<t_3$. By \lemref{lem:add_punctures}, the punctures $t_2$ and $t_3$ do not change the value of the extremal length and we get
\begin{align*}
\el(\alpha, \CHAT \setminus \{ \pm a, \pm i a, \pm 1/a, \pm i/a \}) &= 2 \, \el(\beta,\CC \setminus \{0, \pm 1/a^2, \pm a^2 \})\\
&= 2 \, \el(\gamma,\CC \setminus \{0, 1, t_1, t_2,  t_3 \}) \\
&= 2 \, \el(\gamma,\CC \setminus \{0, 1, t_1\}) \\
&= 4\, (K/K')(1/\sqrt{t_1}).
\end{align*}
\texttt{SageMath} provides the rigorous enclosure $[4.1335929781133 \pm  2.81\times 10^{-14}]$ for the last expression (see the ancillary file \href{https://arxiv.org/src/2404.00336/anc}{\texttt{integrals}}).
\end{proof}

\begin{rem} \label{rem:staple}
We define a \emph{staple curve} on the cube as a curve that surrounds the union of a face diagonal together with the two edges that are adjacent to the endpoints of the diagonal but do not belong to the face it is contained in (this looks like a staple straddling the cube, see \figref{fig:staple}). One staple curve can be represented as the curve surrounding $[-a,a]$ in $\CHAT \setminus \{ \pm a, \pm i a, \pm 1/a, \pm i/a \}$. Similar calculations as in the above lemma show that the extremal length of this curve is
\[
4\, (K / K')(1/\sqrt{t}) \in  [6.9282032302755 \pm 2.74\times 10^{-14}]
\]
where $t = 1+1/a^4$ and $a = (1+\sqrt{3})/\sqrt{2}$. We will not need this calculation to prove \thmref{thm:main}.
\end{rem}


\subsection{The hexagon curves}

We now consider the \emph{hexagon curves}, defined as those obtained by intersecting the cube with a plane that bisects a main diagonal (joining diametrically opposite vertices) perpendicularly. This forms a regular hexagon on the surface of the cube (see \figref{fig:hexagon}). Topologically, a hexagonal curve separates the three edges adjacent to a vertex from the three edges adjacent to the opposite vertex. 

We will not need the results of this subsection to prove \thmref{thm:main} either, but the calculations are nice, and the resulting extremal length is the third smallest that we have found.

Note that hexagon curves display a three-fold symmetry. We thus want to apply a M\"obius transformation to put the fixed points of this rotation at $0$ and $\infty$. Equivalently, before applying the stereographic projection, we want to put two vertices of the cube at $(0,0,\pm1)$. These two points are then sent to $0$ and $\infty$, and interestingly, the images of the remaining six vertices form an equilateral triangle.

\begin{lem}  \label{lem:triangle}
If a cube is inscribed in the unit sphere in such a way that two of its vertices lie at $(0,0,\pm1)$, then the stereographic projection sends the remaining six vertices to the vertices and the midpoints of the sides of an equilateral triangle in the plane.
\end{lem}

\begin{figure}
\centering
\includegraphics[width=0.5\textwidth]{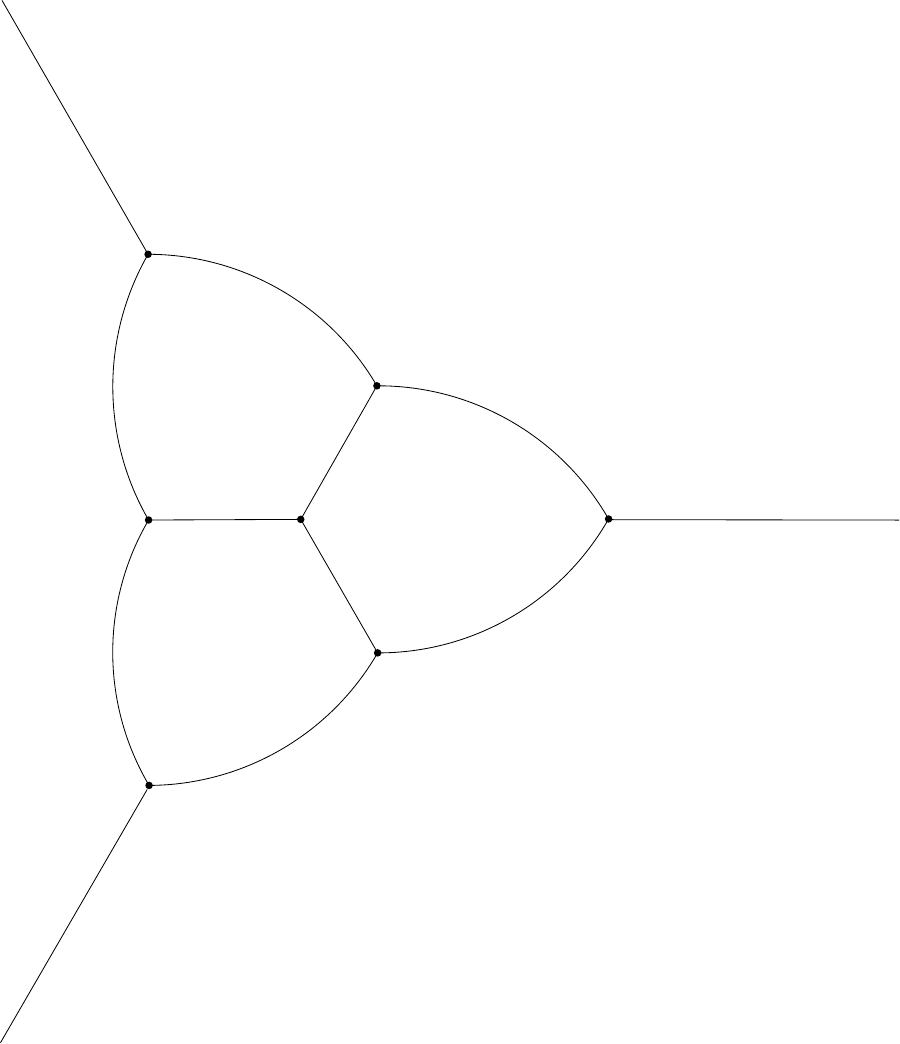}
\caption{A stereographic projection of the cube displaying its $3$-fold symmetry.}
\label{fig:equilateral}
\end{figure}

\begin{proof}
Rotate the cube so that one vertex $v$ has $y$-coordinate equal to zero and the other two coordinates positive. The segment between $v$ and $(0,0,1)$ is an edge of the cube while the segment from $v$ to $(0,0,-1)$ is the diagonal of a face, hence their lengths are in ratio of $1$ to $\sqrt{2}$. A little algebra then shows that $v = (2\sqrt{2}/3,0,1/3)$. The image of $v$ by the stereographic projection is then equal to
\[
S(v) = S(2\sqrt{2}/3,0,1/3) = \frac{2\sqrt{2}/3}{1- 1/3} = \sqrt{2}.
\]
The diametrically opposite vertex $-v$ is sent to
\[
S(-v)=\frac{-1}{S(v)} = -\frac{1}{\sqrt{2}}
\]
since $S$ conjugates the antipodal map to $z\mapsto -1/z$.

By the order 3 rotational symmetry of the cube around the $z$-axis, the other vertices are sent to $\sqrt{2} e^{\pm \frac{2\pi i}{3}}$ and $\frac{1}{\sqrt{2}} e^{\pm \frac{\pi i}{3}}$. We then compute
\[
\frac{\sqrt{2} e^{\frac{2\pi i}{3}}+ \sqrt{2} e^{-\frac{2\pi i}{3}}}{2} = \sqrt{2}\cos(2\pi/3)=-\frac{\sqrt{2}}{2}= -\frac{1}{\sqrt{2}}
\]
so that $S(-v)$ is indeed the midpoint of a side of the equilateral triangle with vertices $\sqrt{2}$, $\sqrt{2} e^{\frac{2\pi i}{3}}$, and $\sqrt{2} e^{\frac{-2\pi i}{3}}$. Then $\frac{1}{\sqrt{2}} e^{\pm \frac{\pi i}{3}}$ are the midpoints of the other two sides by rotational symmetry.
\end{proof}

The images of the edges of the cube by the radial projection onto the sphere followed by the stereographic projection onto the plane are shown in \figref{fig:equilateral}. We then use this configuration to compute the extremal length of the hexagon curves.

\begin{prop}
The extremal length of any hexagon curve on the cube minus its vertices is equal to
\[
6 (K / K')(1/3) \in [3.83778471351302 \pm 9.40\times 10^{-15}].
\]
\end{prop}
\begin{proof}
The previous lemma  together with \lemref{lem:radial} shows that there is a conformal map from the cube to the Riemann sphere that sends the vertices to $\left\{0,\infty,-\frac{1}{\sqrt{2}},\frac{1}{\sqrt{2}} e^{\pm \frac{\pi i}{3}},\sqrt{2},\sqrt{2} e^{\pm\frac{2\pi i}{3}}\right\}$ and sends a hexagon curve to a curve homotopic to the unit circle.

We now apply the map $f(z) = (\sqrt{2}z)^3+1$, which sends the punctures (and critical points) to $0$, $1$, $9$, and $\infty$, and the unit circle to the circle $\beta$ of radius $2\sqrt{2}$ centered at $1$. Since $1<1+2\sqrt{2}< 9$, this circle separates $(0,1)$ from $9$ and $\infty$. By \lemref{lem:el_under_covers} and \lemref{lem:formule4pt}, the extremal length of the hexagon curve on $X$ is 
\[
3 \el(\beta, \CC \setminus \{ 0, 1, 9\}) = 6 \, (K/K')(1/3).
\]
\texttt{SageMath} provides the rigorous enclosure $[3.83778471351302 \pm 9.40\times 10^{-15}]$ for this expression (see the ancillary file \href{https://arxiv.org/src/2404.00336/anc}{\texttt{integrals}}).
\end{proof}

\begin{rem}
We define \emph{tripod curves} as the curves that surround the union of the three segments from the origin to $\sqrt{2}$, $\sqrt{2} e^{\frac{2\pi i}{3}}$ and $\sqrt{2} e^{-\frac{2\pi i}{3}}$ (see Fi\-gure \ref{fig:tripod} for what this corresponds to on the cube). A similar calculation as above shows that the extremal length of this tripod curve is
\[
6 (K/ K')(2 \sqrt{2}/3) \in [9.3804115361767 \pm 5.12\times 10^{-14}].
\]
 In fact, a similar argument as in \cite[p.24]{BolzaEL} shows that the pro\-duct of the extremal lengths of the hexagon and tripod curves is equal to $36=6^2$. This is an example where Minsky's inequality (see \secref{sec:Minsky}) is sharp.
\end{rem}

\section{The edge curves} \label{sec:edge}

In this section, we prove a characterization of the quadratic differentials that realize the extremal length of the \emph{edge curves}, defined as the curves that surround any edge of the cube. For this, we first quotient by a rotation of order $2$, which reduces the calculation to computing the extremal length of a certain curve on a six-times-punctured sphere. Because the six punctures do not all lie on a circle, this extremal length cannot be calculated as a ratio of complete elliptic integrals of the first kind. However, we still manage to show that the quadratic differential realizing the extremal length belongs to a certain $1$-parameter family, and then find a criterion to determine the correct parameter. Solving for this parameter amounts to solving a transcendental equation, which is why we do not obtain an exact formula in the end. We can nevertheless estimate the parameter and the resulting extremal length rigorously using interval arithmetic.

\subsection{Quotienting by the rotational symmetry}

The first step is to apply a conformal transformation to highlight the order $2$ rotational symmetry of the edge curves.

\begin{lem}  \label{lem:edge_configuration}
There is a conformal map between the cube minus its vertices and $\CHAT \setminus \{\pm \rho, \pm 1/\rho, \pm \omega , \pm \overline{\omega}\}$, where
\[
\rho = \sqrt{3}-\sqrt{2}
\]
and
\[
\omega = \frac{2\rho  + i(1-\rho^2)}{1+\rho^2}=\frac{1 + i\sqrt{2}}{\sqrt{3}}.
\] 
Furthermore, this map sends an edge of the cube to the interval $(-\rho,\rho)$.
\end{lem}

\begin{proof}
By \lemref{lem:plus_sign} and a further homothety of factor $a = (1+\sqrt{3})/\sqrt{2}$, there is a conformal map $g$ from the cube to $\CHAT$ that sends the vertices to  $\{\pm 1, \pm i, \pm a^2, \pm i a^2 \}$.

For $\rho \notin \{ -1, 1, \infty\}$, the M\"obius transformation 
\[
f(z) = \left(\frac{1+\rho}{1-\rho}\right) \left(\frac{z-1}{z+1}\right)
\]
satisfies $f(1)=0$, $f(1/\rho)=1$, $f(-1)=\infty$, and $f(-1/\rho)=\left(\frac{1+\rho}{1-\rho} \right)^2$. We want to choose $\rho$ such that $f(-1/\rho)=a^2$. For this, we set
\[
\frac{1+\rho}{1-\rho} = a,
\]
which is equivalent to
\[
\rho = \frac{a - 1}{a+1} = \frac{1+\sqrt{3}-\sqrt{2}}{1+\sqrt{3}+\sqrt{2}}.
\]
We can also write
\begin{align*}
\rho &= \left(\frac{1-\sqrt{2}+\sqrt{3}}{1+\sqrt{3}+\sqrt{2}}\right)\left(\frac{1-\sqrt{2}-\sqrt{3}}{1-\sqrt{3}-\sqrt{2}}\right)\\ & = \frac{2\sqrt{2}}{4+2\sqrt{6}}
 = \frac{1}{\sqrt{2}+\sqrt{3}}= \sqrt{3}-\sqrt{2}.
\end{align*}

We then compute $\rho^2 = 5-2\sqrt{6}$,
\[
1+\rho^2 = 6 - 2\sqrt{6} = 2\sqrt{3} \, \rho
\]
and
\[
1-\rho^2 = 2\sqrt{6}-4 = 2\sqrt{2} \, \rho,
\]
so that
\[
\frac{2\rho  + i(1-\rho^2)}{1+\rho^2} = \frac{1 + i\sqrt{2}}{\sqrt{3}}.
\]

Elementary calculations then show that $f(\rho)=-1$, $f(-\rho) = -a^2$, $f(\omega) = i$, $f(\overline{\omega}) = -i$, $f(-\omega)=i a^2$, and $f(-\overline{\omega})=-i a^2$. The desired conformal map is thus given by $f^{-1}\circ g$. 

Since $f$ is real on the real line and its pole is at $-1 < -\rho$, we have that $f((-\rho,\rho)) = (-a^2,-1)$. It is easy to see that this interval is also the image of an edge of the cube by $g$.
\end{proof}

We can then express the extremal length of an edge curve in terms of another curve on a six-times-punctured sphere instead of an eight-times-punctured one.

\begin{cor} \label{cor:reduction_to_6}
The extremal length of any edge curve on the cube minus its vertices is equal to $2 \el(\beta, \CC \setminus \{ 0, \rho^2, 1/\rho^2, \omega^2, \overline{\omega}^2\})$ where $\rho$ and $\omega$ are as in \lemref{lem:edge_configuration} and  $\beta$ is a curve surrounding $[0,\rho^2]$.
\end{cor}
\begin{proof}
Let $\alpha$ be a curve which surrounds $[-\rho, \rho]$ in $\CHAT \setminus \{\pm \rho, \pm 1/\rho, \pm \omega , \pm \overline{\omega}\}$ and is invariant under the rotation $z\mapsto -z$. For instance, we can take $\alpha$ to be the circle of radius $\sqrt{\rho}$ centered at the origin (note that $\rho < 1 = |\omega|^2$, so this circle does not enclose the non-real punctures). By \lemref{lem:edge_configuration}, the extremal length of an edge curve on the cube punctured at its vertices is equal to the extremal length of $\alpha$ in  $\CHAT \setminus \{\pm \rho, \pm 1/\rho, \pm \omega , \pm \overline{\omega}\}$.

We then apply the squaring map $f(z)=z^2$, which sends the punctures and critical points to $0$, $\rho^2$, $1/\rho^2$, $\omega^2$, $\overline{\omega}^2$, and $\infty$. We also have that $f(\alpha)$ traces the circle $\beta$ of radius $\rho$ centered at the origin twice, so that $f^{-1}(\beta)=\alpha$. By \lemref{lem:el_under_covers}, we have 
\[
\el(\alpha,\CHAT \setminus \{\pm \rho, \pm 1/\rho, \pm \omega , \pm \overline{\omega}\}) = 2 \el(\beta, \CC \setminus \{ 0, \rho^2, 1/\rho^2, \omega^2, \overline{\omega}^2\})
\]
as required.
\end{proof}

\begin{rem} \label{rem:algebra}
Recall that $\rho^2 = 5-2\sqrt{6}$, so that $1/\rho^2 = 5 + 2\sqrt{6}$. We also have
\[
\omega^2 = \left( \frac{1 + i \sqrt{2}}{\sqrt{3}} \right)^2 = \frac{-1+i \,2\sqrt{2}}{3}.
\]
This will be useful in subsection \ref{subsec:numerical}.
\end{rem}

As mentioned before, the fact that $\omega^2$ and $\overline{\omega}^2$ are not real implies that we cannot express $\el(\beta, \CC \setminus \{ 0, \rho^2, 1/\rho^2, \omega^2, \overline{\omega}^2\})$ simply in terms of elliptic integrals. Moreover, there is no further rotational symmetry to exploit. Indeed, $\{0, \rho^2, 1/\rho^2,\omega^2,\overline{\omega}^2,\infty\}$ is invariant under the involution $z \mapsto 1/z$, but the latter does not preserve the homotopy class of $\beta$. If we instead map \[\{0, \rho^2, 1/\rho^2, \infty\}\] to four points that are symmetric in pairs about the origin (with the images of $0$ and $\rho^2$ in one pair), then $\omega^2$ and $\overline{\omega}^2$ do no land on the imaginary axis, so we cannot square again.

Although $\beta$ and  $\CC \setminus \{ 0, \rho^2, 1/\rho^2, \omega^2, \overline{\omega}^2\}$ admit a reflection symmetry across the real axis, quotienting by this symmetry does not simplify things much. Instead of doing that, we simply exploit the fact that the quadratic differential realizing the extremal length of $\beta$ must be invariant under that symmetry. To determine this quadratic differential, we make an educated guess for what it should look like, then progressively refine our guess in the next two subsections.

\subsection{A 1-parameter family of Jenkins--Strebel differentials}

Our goal is to identify the quadratic differential realizing the extremal length of the curve $\beta$ in \corref{cor:reduction_to_6}. However, all of the arguments in this subsection apply to the following more general situation. Suppose that $A < B < C$ and $\im W > 0$. Let $\beta$ be a curve surrounding $[A,B]$ in \[\CC \setminus \{ A, B, C, W, \wbar{W} \},\] and let $q_\beta$ be the quadratic differential realizing the extremal length of $\beta$.

Since we are working on the Riemann sphere, we know that $q_\beta = f_\beta(z) dz^2$ for some rational function $f_\beta$. Since any rational function is determined up to a constant by its zeros and its poles (with multiplicities), we need to locate these singularities.

We know that $q_\beta$ has at most $6$ simple poles at the punctures. Indeed, $q_\beta$ is required to be holomorphic on the surface $\CHAT \setminus \{ A, B, C, W, \wbar{W},\infty\}$ and to be integrable, so its poles must be simple. Furthermore, by the Euler--Poincaré formula, the number of zeros minus the number of poles of $q_\beta$ is equal to $-4$. This can also be shown by observing that $dz^2$ has no zeros and a pole of order $4$ at infinity. Moreover, the ratio of two quadratic differentials is a meromorphic function (hence a rational function in this case), which has the same number of zeros as poles on a closed surface, namely, its degree. We conclude that $q_\beta$ has at most two zeros counting multiplicity.

Since the domain $\CHAT \setminus \{ A, B, C, W, \wbar{W},\infty\}$ and the curve $\beta$ are symmetric about the real axis, the uniqueness part of Jenkins's theorem implies that $q_\beta$ is also symmetric about the real axis, so that its zeros and its poles are either real or come in conjugate pairs. Thus either $q_\beta$ has a pair of conjugate non-real zeros, two simple real zeros, one double real zero, one simple real zero (in which case one of the punctures is not a pole), or no zero at all (in which case two of the punctures are not poles). We rule out some of these possibilities by considering the structure of the critical graph of $q_\beta$. By definition, the \emph{critical graph} of a quadratic differential is the union of the critical trajectories together with the singularities they limit to.

\begin{lem}
If $P \subset \CHAT$ is a finite set and $q$ is a Jenkins--Strebel quadratic differential on $\CHAT \setminus P$ with only one cylinder, then the critical graph of $q$ forms a pair of topological trees whose leaves are the simple poles of $q$, hence contained in $P$. 
\end{lem}
\begin{proof}
Let $U$ be the union of the regular horizontal trajectories of $q$. By hypothesis, $U$ is an annulus, so its complement in $\CHAT$ has two connected components. Now $G = \CHAT \setminus U$ is precisely the critical horizontal graph of $q$. If $G$ contained a cycle, then its complement $U$ would not be connected by the Jordan curve theorem. Hence $G$ is a forest, and since it has two components, it is a pair of trees.
The leaves of the critical graph are $1$-prong singularities, which correspond to simple poles. Since $q$ is holomorphic in $\CHAT \setminus{P}$, these are contained in $P$.
\end{proof}

We use this to show that $q_\beta$ has a unique zero in $[- \infty, A) \cup ( B, +\infty)$.

\begin{lem}  \label{lem:unique_zero}
Let $q_\beta$ be the quadratic differential realizing the extremal length of the curve $\beta$ surrounding $[A,B]$ in $\CC \setminus \{ A, B, C, W, \wbar{W}\}$, where $A<B<C$ and $\im W >0$. Then $q_\beta$ has simple poles at $A$, $B$, $W$, and $\wbar W$, and either
\begin{itemize}
\item simple poles at $C$ and $\infty$, and a double zero in $(C , + \infty)$;
\item a simple pole at $C$ but not at $\infty$, and a simple zero in $[-\infty,A)$;
\item or a simple pole at $\infty$ but not at $C$, and a simple zero in $(B,C]$.
\end{itemize}
\end{lem}
\begin{proof}
Let $T_1$ and $T_2$ be the two trees in the critical graph $G$ of $q_\beta$.  separated by the regular trajectories. By definition, the regular horizontal trajectories of $q_\beta$ are homotopic to $\beta$, so one of the trees, say $T_1$, has its leaves in the set $\{A, B\}$ while $T_2$  has its leaves in the set $\{C, W, \wbar{W}, \infty\}$, because $\beta$ separates these two sets of punctures. Now, a tree with only two leaves is simply an edge, and by the invariance of $G$ under complex conjugation, we deduce that $T_1 = [A,B]$. In particular, $q_\beta$ has simple poles at $A$ and $B$.

Since $T_2$ is a tree, there is a unique shortest (or non-backtracking) path $\gamma$ in $T_2$ between any two of its points. Let $\gamma$ be the shortest path between $C$ and $\infty$ in $T_2$. By the uniqueness of $\gamma$ and the invariance of $T_2$ under complex conjugation, $\gamma$ must be equal to $[C,+\infty]$. By a similar argument, the shortest path $\eta$ between $W$ and $\wbar{W}$ in $T_2$ is preserved by complex conjugation, so it is an arc that crosses $\RR \cup \{\infty \}$ in a single point that we call $r_0$. 

Note that $W$ and $\wbar{W}$ are necessarily leaves of $T_2$. Otherwise we could extend $\eta$ to a longer trajectory ending in a conjugate pair of (distinct) leaves, but $W$ and $\wbar{W}$ is the only conjugate pair of distinct punctures. We conclude that $q_\beta$ has simple poles at $W$ and $\wbar{W}$.

If $r_0 \in (C,+\infty)$, then $q_\beta$ has a double zero at $r_0$ because there are at least $4$ horizontal trajectories emanating from that point (the union of $\gamma$ and $\eta$) and we observed earlier that there cannot be more than $2$ zeros counting multiplicity. In that case, $q_\beta$ must have simple poles at $C$ and $\infty$ for the number of zeros minus the number of poles to be equal to $-4$ counting multiplicity.

By considering the shortest path from $r_0$ to $C$ in $T_2$, we see that if $r_0<A$ then $(-\infty,r_0] \subset T_2$ and if $r_0\in (B,C]$ then $[r_0,C] \subset T_2$. In either case, since there are at least $3$ horizontal trajectories emanating from $r_0$ (a real interval together with $\eta$), the quadratic differential $q_\beta$ has a zero at $r_0$. Furthermore, in these cases, the zero must be simple. Otherwise, there would be a fourth trajectory emanating from $r_0$ along the real axis, but there is nowhere for this trajectory to end. Indeed, there is no other real zero because from such a zero would leave a symmetric pair of trajectories not contained in the real axis. This pair of trajectories cannot end in $W$ and $\overline{W}$ since $\eta$ is the unique trajectory ending in these points, and it cannot return to the real axis elsewhere either since $T_2$ is a tree. Hence if $r_0<A$, then a putative fourth trajectory emanating from $r_0$ would necessarily end at the next pole $A$, which is not possible since the only trajectory emanating from that point is $(A,B)$. A similar argument rules out a fourth trajectory if $r_0\in (B,C]$.

This covers all the possibilities since $r_0$ cannot belong to $[A,B]$ because $T_1$ and $T_2$ are disjoint.
\end{proof}

This leads us to define the following $1$-parameter family of quadratic differentials.

\begin{defn}
Let
\[
f_r(z) := \begin{cases} \displaystyle\frac{1}{(z-A)(z-B)(z-C)(z-W)(z-\overline{W})} & \text{if } r = -\infty \\
\displaystyle\frac{z-r}{(z-A)(z-B)(z-C)(z-W)(z-\overline{W})} & \text{if } r\in(-\infty,A) \\
\displaystyle \frac{z-r}{(z-A)(z-B)(z-W)(z-\overline{W})} & \text{if }r\in (B,C] \\
\displaystyle\frac{(z-r)^2}{(z-A)(z-B)(z-C)(z-W)(z-\overline{W})} & \text{if }r\in (C,+\infty)
\end{cases}
\]
and let $q_r = f_r(z)dz^2$.
\end{defn}

The previous lemma implies that $q_\beta$ belongs to this $1$-parameter family.

\begin{cor}
There exists a unique $r_0 \in [- \infty, A) \cup ( B, +\infty)$ such that $q_\beta = \lambda q_{r_0}$ for some $\lambda>0$.
\end{cor}
\begin{proof}
Let $r_0  \in [- \infty, A) \cup (B, +\infty)$ be the zero of $q_\beta$ provided by \lemref{lem:unique_zero} and let $f_\beta$ be the rational function such that $q_\beta = f_\beta(z)dz^2$. \lemref{lem:unique_zero} tells us the order of the zero of $q_\beta$ at $r_0$ and the location of the simple poles, depending on which interval $r_0$ belongs to among $[-\infty, A)$, $(B, C]$, and $(C,+\infty)$. Now, the finite zeros and poles of $f_\beta$ are the same as those of $q_\beta$, which also coincide with those of $f_{r_0}$ by \lemref{lem:unique_zero} and the definition of $f_r$. It follows that the rational function $f_\beta / f_{r_0}$ has no zeros nor poles in $\CC$, hence is equal to a constant $\lambda \in \CC \setminus \{ 0 \}$. Furthermore, since $q_\beta$ is symmetric about the real axis, it is real when evaluated at vectors tangent to the real axis, which is to say that $f_\beta$ is real along the real axis. By construction, $f_{r_0}$ is also real along the real axis since $(z-W)(z-\overline{W})=z^2-2\re(W) z + |W|^2$ and the other factors are all real. We conclude that $\lambda \in \RR \setminus \{0 \}$. Lastly, we know from \lemref{lem:unique_zero} that $(A,B)$ is a horizontal trajectory of $q_\beta$, which means that $f_{\beta}$ is positive on that interval. It is easy to check that the same holds for $f_r$ for any $r$, hence $\lambda = f_\beta / f_{r_0} > 0$. 
\end{proof}

Note that since $q_\beta$ itself is only defined up to a positive constant anyway, we may as well say that $q_\beta = q_{r_0}$.

\subsection{Finding the correct parameter}

In this subsection, we devise a criterion to determine the parameter $r_0$ such that $q_\beta = q_{r_0}$. We first observe that the quadratic differentials $q_r$ are all Jenkins--Strebel for topological reasons.

\begin{lem} \label{lem:cylinders}
For every $r \in [-\infty,A) \cup (B,+\infty)$, the quadratic differential $q_r$ is Jenkins--Strebel with at most two cylinders. 
\end{lem}
\begin{proof}
 In all cases, $(A,B)$ and $(C,+\infty)$ are horizontal because $f_r > 0$ there and $dz^2$ is positive when evaluated on vectors tangent to the real line (i.e., on real vectors). The critical trajectory of $q_r$ emanating from $W$ must intersect $\RR \cup \{ \infty\}$ since there are no other zeros or poles in the open upper half-plane. Together with its reflection, this trajectory forms an arc $\eta$ that connects $W$ and $\wbar{W}$.
 
 Cut the Riemann sphere along  $[A,B]$, $[C,+\infty]$ and $\eta$. The resulting surface $S$ is an annulus if $\eta \cap [C,+\infty] \neq \varnothing$ and a pair of pants otherwise (since $q_r$ has no zeros in $[A,B]$, this interval is disjoint from $\eta$). Furthermore, the restriction of $q_r$ to $S$ is holomorphic and extends to be analytic and non-negative along the ideal boundary of $S$, since cutting along the critical trajectory from a simple pole amounts to taking a square root, which gets rid of the pole. 
 
 It is well-known that any holomorphic quadratic differential on an annulus that is po\-si\-ti\-ve along the boundary only has closed trajectories that are all homotopic to each other (see e.g. \cite[Theorem 9.4]{Strebel}). Thus, if $S$ is an annulus, then $q_r$ is Jenkins--Strebel with a single cylinder. In that case, since the regular horizontal trajectories of $q_r$ surround $[A,B]$, they are homotopic to $\beta$, so that $q_r= q_\beta$ and hence $r=r_0$.
 
 Suppose instead that $S$ is a pair of pants. On a pair of pants, any quadratic differential $\psi$ that is positive along the boundary is Jenkins--Strebel with at most $3$ cylinders \cite{JenkinsPants} (see also \cite[Proposition 6.8]{FLP}). If there are three cylinders, then $\psi$ has no zeros on the boundary, hence has two zeros in the interior counting multiplicity. This can be seen by sewing the three ends of the pair of pants shut to obtain a quadratic differential on the sphere with six simple poles (where the boundary components were folded in half). Since the number of poles minus the number of zeros of a quadratic differential on the sphere is $4$, the total number of zeros in the interior of the pair of pants must be equal to $2$. Alternatively, this follows from the Euler--Poincaré formula for surfaces with boundary.
 
 By construction, $q_r$ has a unique zero and if it is double, then it lies in $(C,+\infty)$ and not in the interior of $S$, so it cannot have $3$ cylinders by the previous paragraph. We conclude that there are at most two cylinders in this case.
\end{proof}

In order to determine the parameter $r_0$, it thus suffices to determine when $q_r$ has exactly one horizontal cylinder rather than two. We need one more lemma before we can give a criterion to detect the presence of a second cylinder.

\begin{lem} \label{lem:twice}
For every $r \in [-\infty,A) \cup (B,+\infty)$, each regular horizontal trajectory of $q_r$ intersects the real axis exactly twice.
\end{lem}
\begin{proof}
Let $\gamma$ be a regular horizontal trajectory of $q_r$. By \lemref{lem:cylinders}, $\gamma$ is closed, and is simple (every trajectory is).  By the Jordan--Schoen\-flies theorem, $\gamma$ divides the sphere into two topological disks $D_1$ and $D_2$. By doubling $D_j$ across its boundary, we see that the number of poles minus the number of zeros of $q_r$ inside $D_j$ is equal to $2$ counting multiplicity (this also follows from the Euler--Poincaré formula). This implies that $\gamma$ intersects $\RR\cup \{ \infty\}$. Otherwise, it would be contained in the upper or the lower half-plane and thus the disk that it bounds in that half-plane would contain at most one simple pole ($W$ or $\wbar{W}$).

Since $\RR\cup \{ \infty\}$ disconnects the sphere, any simple closed curve that intersects it transversely must intersect it another time. Since $q_r$ is real along $\RR$, regular trajectories that intersect $\RR$ are either contained in $\RR$ or intersect it transversely (where $q_r$ is negative). Since $\RR\cup \{ \infty\}$ itself is not a closed horizontal trajectory (and does not contain any), we conclude that $\gamma$ intersects $\RR\cup \{ \infty\}$ at least twice.

The other important observation is that $\gamma$ is symmetric about the real axis. Indeed, since $q_r$ is invariant under complex conjugation, the reflection $\overline{\gamma}$ is a horizontal trajectory passing through the same points as $\gamma$ on the real axis, but since there is only one horizontal trajectory through any regular point of $q_r$ we deduce that $\overline{\gamma} = \gamma$. Since $\gamma$ is simple and connected, it follows that $\gamma \cap (\RR\cup \{ \infty\})$ contains only two points. 

Finally, note that $\gamma$ does not pass through $\infty$ because $\infty$ is either a critical point of $q_r$ or a regular point where the horizontal trajectory passes through along $\RR$ rather than transverse to $\RR$ (this happens when $r\in(-\infty,A)$, in which case $f_r$ is positive near $\pm \infty$ on the real line).
\end{proof}

This lemma implies that each horizontal cylinder of $q_r$ intersects $\RR$ in a pair of intervals. Furthermore, these two intervals have the same length with respect to the metric $\sqrt{|q_r|}$ since this measures the height of the cylinder. We can now state our criterion to find $r_0$.

\begin{lem}\label{lem:transverse_measure}
If $r \in [-\infty, A)$, then $r=r_0$ if and only if
\[
\int_r^A \sqrt{|q_r|} = \int_{B}^{C} \sqrt{|q_r|}.
\]
If $r \in (B,C]$, then $r = r_0$ if and only if
\[
\int_{-\infty}^A \sqrt{|q_r|} = \int_{B}^r \sqrt{|q_r|}.
\]
If $r \in (C,+\infty)$, then $r=r_0$ if and only if
\[
\int_{-\infty}^A \sqrt{|q_r|} = \int_{B}^{C} \sqrt{|q_r|}.
\]
\end{lem}
\begin{proof}
We only prove the first statement, the other two being very similar. We thus suppose that $r \in [-\infty, A)$, in which case $f_r$ is negative on \[(r,A) \cup (B, C)\] and positive elsewhere on the real axis (minus the singularities). 

Suppose that $r=r_0$. Then $q_r$ has only one horizontal cylinder $U$. By \lemref{lem:twice}, $U$ intersects $\RR$ in two intervals $I_1$ and $I_2$, necessarily contained in $(r,A) \cup (B, C)$. On the other hand, we know from the proof of \lemref{lem:unique_zero} that the only critical horizontal trajectories of $q_\beta=q_r$ are $(A,B)$,  $(C , +\infty) \cup [-\infty, r)$, and the horizontal trajectories from $r$ to $W$ and $\overline{W}$. Hence every horizontal trajectory passing through $(r,A) \cup (B, C)$ is regular, thus belong to $U$. This means that \[I_1 \cup I_2 = (r,A) \cup (B, C)\] and thus $I_1 = (r,A)$ and $I_2 = (B, C)$ up to relabelling. The height of the cylinder $U$ is the length of either $I_1$ or $I_2$ with respect to $\sqrt{|q_r|}$, so that
\[
\int_{I_1} \sqrt{|q_r|} = \int_{I_2} \sqrt{|q_r|}.
\]

Now suppose that $r \neq r_0$, so that $q_r$ has two horizontal cylinders $U_1$ and $U_2$. By the proof of \lemref{lem:cylinders}, one of these two cylinders, say $U_1$, surrounds $[A,B]$. We now claim that all the horizontal trajectories in $U_2$ intersect only one of $(r,A)$ or $(B, C)$. Indeed, if a closed horizontal trajectory $\gamma$ intersects both intervals, then it bounds a disk $D$ that contains the poles $A$ and $B$ but not $C$, and is symmetric about the real line. If $D$ contains one of $W$ or $\overline{W}$, then it also contains the other one by symmetry. In that case, $q_r$ must have two zeros in $D$ counting multiplicity by the Euler--Poincaré formula. However, $q_r$ only has one zero in $\CHAT$ when $r \in [-\infty, A)$ and this zero is simple. It follows that $\gamma$ is homotopic to $\beta$ in $\CC \setminus \{ A,B,C,W, \overline{W}\}$, hence is contained in $U_1$, which proves the claim.

By \lemref{lem:twice}, $U_j \cap ((r,A) \cup (B, C))$ is a union of two disjoint vertical segments $v_j^\pm \subset U_j$, for each $j \in \{1,2\}$. Moreover, the union of these four segments is equal to $(r,A) \cup (B, C)$ minus a finite number of points belonging to critical horizontal trajectories of $q_r$. By the previous paragraph, $(r,A)$ and $(B, C)$ each contain one of $v_1^\pm$ while $v_2^\pm$ are both contained in the same interval. Therefore, the length of one of the intervals $(r,A)$ or $(B, C)$ with respect to $\sqrt{|q_r|}$ is equal to the height of $U_1$ while the length of the other interval is equal to the height of $U_1$ plus twice the height of $U_2$. In particular,
\[
\int_r^A \sqrt{|q_r|} \neq \int_{B}^{C} \sqrt{|q_r|}
\]  
as required.
\end{proof}

The next proposition proceeds to gives a computable formula for the extremal length of $\beta$ when the parameters $A$, $B$, $C$, and $W$ are such that the first possibility occurs in the previous lemma. There are analogous formulas in the other cases as well, but this is the case that will occur for us. To simplify notation, we write
\begin{equation} \label{eq:integral}
I_r(l,u) : = \int_l^u \sqrt{|q_r|}
\end{equation}
and
\begin{equation}
J_r(l,u) : = \int_l^{+\infty} \sqrt{|q_r|} + \int_{-\infty}^{u} \sqrt{|q_r|}
\end{equation}
in the statement below.

\begin{prop} \label{prop:formulas} 
Suppose that $r_0 \in [-\infty, A)$ and let \[E = \el(\beta, \CC \setminus \{ A, B, C, W, \wbar{W} \}).\] Then
\[
E = \inf_{r<A} \frac{2I_r(A,B)}{\min\left( I_r(r,A) , I_r(B,C) \right)}
\]
and
\[
E = \sup_{r<A} \frac{4I_r(A,B)}{I_r(r,A) + I_r(B,C) + \frac{J_r(C,r)}{I_r(A,B)}| I_r(r,A) - I_r(B,C)|},
\]
where the infimum and supremum is achieved if and only if $r=r_0$.
\end{prop}
\begin{proof}
Let $r\in [-\infty, A)$ and let $U_1$ be the horizontal cylinder of $q_r$ surrounding $[A,B]$. The height of $U_1$ is 
\[
H_1 = \min\left( \int_r^A \sqrt{|q_r|},\int_{B}^{C} \sqrt{|q_r|} \right)
\]
by the end of the proof of \lemref{lem:transverse_measure}, while its circumference is $C_1 = 2\int_A^B \sqrt{|q_r|}$ because $[A,B]$ is an end of $U_1$, so there is a sequence of closed horizontal trajectories in $U_1$ which converge to the segment $[A,B]$ traced twice.

We thus have
\[
E \leq \el(U_1) = \frac{2\int_A^B \sqrt{|q_r|}}{\min\left( \int_r^A \sqrt{|q_r|},\int_{B}^{C} \sqrt{|q_r|} \right)}=\frac{2I_r(A,B)}{\min\left( I_r(r,A) , I_r(B,C) \right)}.
\]
By Jenkins's theorem, equality is achieved if and only if $U_1$ is the horizontal cylinder of $q_\beta = q_{r_0}$, which happens only if $r = r_0$.

On the other hand, for every $r\in [-\infty, A)$ we can use the conformal metric $h_r$ with area form $|q_r|$ in the definition of the extremal length of $\beta$. As observed above, the horizontal trajectories in $U_1$ are homotopic to $\beta$ and have length $C_1=2 I_r(A,B)$. Since closed trajectories minimize length in their homotopy class, we get $\ell([\beta],h_r) = C_1$. It remains to find an upper bound on $\area(h_r)$. 

Suppose that $r \neq r_0$ and let $U_2$ be the other horizontal cylinder of $q_r$. Let $H_2$ be the height of $U_2$ and $C_2$ its circumference. Recall from the proof of \lemref{lem:transverse_measure} that the shorter of the two intervals $(A,r)$ and $(B,C)$ with respect to $h_r$ has length $H_1$ and the other one has length $H_1+2H_2$, so that
\[
H_2 = \frac{\left| \int_r^A \sqrt{|q_r|} -\int_{B}^{C} \sqrt{|q_r|} \right|}{2}.
\]

Observe that $(C,+\infty)\cup [-\infty,r)$ is a horizontal trajectory of $q_r$ en\-ding in a $1$-prong at $C$, but there are two other horizontal trajectories $\gamma^\pm$ incident to the simple zero at $r$, symmetric about the real line. These trajectories must meet the real axis elsewhere for otherwise they would terminate at the simple poles $W$ and $\overline{W}$, but in that case $(A,B)$,  $(C,+\infty)\cup [-\infty,r)$,  $\gamma^+$, and $\gamma^-$ account for all the critical horizontal trajectories, yet the complement of the union of their closures is ho\-meo\-morphic to a cylinder, so that $r=r_0$. We conclude that  $\gamma^+$ and $\gamma^-$ intersect at some regular point of $q_r$ on the real axis, so they coincide and form the outer boundary $\gamma$ of one of $U_1$ or $U_2$. Now, $[C,+\infty)\cup [-\infty,r]\cup \gamma$ is a connected component of the critical graph, so it is an end of the other horizontal cylinder, whose closed horizontal trajectories limit onto this set with the arc $[C,+\infty)\cup [-\infty,r]$ traced twice and $\gamma$ traced once. It follows that these trajectories have length 
\[
\max(C_1,C_2) = 2\left(\int_{-\infty}^r \sqrt{|q_r|} + \int_{C}^{+\infty} \sqrt{|q_r|} \right) + \min(C_1,C_2), 
\]
so that
\[
|C_1-C_2| = 2\left(\int_{-\infty}^r \sqrt{|q_r|} + \int_{C}^{+\infty} \sqrt{|q_r|} \right).
\]

We then estimate 
\begin{align*}
\area(h_r) &= C_1 H_1 + C_2 H_2 \\
& \leq C_1 H_1 + (C_1 + |C_1-C_2|)H_2 \\
& = C_1(H_1 + H_2) + |C_1-C_2| H_2 \\
& = 2 I_r(A,B) \left(\frac{I_r(r,A)+I_r(B,C)}{2}\right) \\
& \quad + J_r(C,r)|I_r(r,A)-I_r(B,C)|
\end{align*}
and this remains valid if $r=r_0$ by continuity.

The definition of extremal length then yields
\begin{align*}
E &\geq \sup_{r<A} \frac{\ell([\beta],h_r)^2}{\area(h_r)} \\
&\geq \sup_{r<A} \frac{4I_r(A,B)}{I_r(r,A) + I_r(B,C) + \frac{J_r(C,r)}{I_r(A,B)}| I_r(r,A) - I_r(B,C)|}.
\end{align*}
If $r=r_0$, then $I_r(r,A) = I_r(B,C)$ by \lemref{lem:transverse_measure} so that the last supremand simplifies to $\displaystyle \frac{2I_r(A,B)}{I_r(r,A)}$, which is the ratio of circumference to height of the horizontal cylinder of $q_{\beta}$, hence equal to $E$. Conversely, if the last supremand is equal to $E$ for some $r$, then we get $\displaystyle E = \frac{\ell([\beta],h_r)^2}{\area(h_r)}$ as well, hence $q_r = q_\beta$ by the uniqueness part of Jenkins's theorem, so that $r = r_0$.
\end{proof}

\subsection{Numerical results}  \label{subsec:numerical}

We finally specialize to the case $A = 0$, $B = \rho^2=5 - 2\sqrt{6}$, $C = 1/\rho^2 = 5+2\sqrt{6}$ and $\displaystyle W = \omega^2= \frac{-1+i2\sqrt{2}}{3}$ (where $\rho$ and $\omega$ are as in \lemref{lem:edge_configuration}) and estimate the extremal length of $\beta$ in this specific case.

We first show that $r_0$ is indeed negative in this case and give an estimate for its value. To that end,  for $r<0$ define 
\[\delta(r) := \int_{r}^0 \sqrt{|q_r|} - \int_{\rho^2}^{1/\rho^2} \sqrt{|q_r|}.\]
This function is clearly continuous, so to prove that it vanishes somewhere, it suffices to show that it changes sign somewhere. To do this, we use interval arithmetic to estimate integrals rigorously and thereby give certified bounds on $r_0$.

Recall that $q_r = f_r(z)dz^2$ where
\begin{align*}
f_r(z) &= \frac{z-r}{(z-A)(z-B)(z-C)(z-W)(z-\overline{W})} \\
& = \frac{z-r}{z(z-B)(z-C)(z^2-2\re(W)z + |W|^2)} \\
& = \frac{z-r}{z(z-B)(z-C)(z^2+(2/3)z + 1)} \\
& = \frac{z-r}{z(z^2-10z+1)(z^2+(2/3)z + 1)}
\end{align*}
when $r<0$. Up to a factor of $3$, this is the formula that appears in \thmref{thm:main} stated in the introduction. We can now estimate $r_0$.

\begin{lem} \label{lem:r_0}
We have $\delta(-9.0795)>0>\delta(-9.0791)$ and hence \[r_0 \in (-9.0795,-9.0791).\]
\end{lem}
\begin{proof}
The \texttt{Arb} package in \texttt{SageMath} can evaluate integrals of analytic functions up to certified error bounds using interval arithmetic. It cannot evaluate improper integrals of the type we are after, so we change the bounds of integration slightly and the estimate the errors by hand.

We write \[\delta(r) = I_r(r,0)-I_r(B,C)\] with $I_r$ as in \eqnref{eq:integral} to simplify notation.

We first seek to bound $\delta$ from below at $r = -9.0795$. Let $\eps > 0$ be small enough so that $r+\eps < -\eps$, $\eps < B-\eps$, and $B+\eps < C-\eps$ (we take $\eps = 10^{-12}$ in our program).  We have $I_r(r,0) > I_r(r+\eps,-\eps)$ and the latter integral can be bounded below rigorously with the \texttt{Arb} package. We also have
\[
I_r(B,C) = I_r(B,B+\eps)+I_r(B+\eps,C-\eps)+I_r(C-\eps,C).
\]
The middle term is handled with the \texttt{Arb} package, while the other two terms can be estimated as follows.

Recall that
\[
I_r(l,u) = \int_l^u \sqrt{|f_r(z)|}dz = \int_l^u \sqrt{\frac{|z-r|}{|z||z-B||z-C||z^2+(2/3)z + 1|}}dz .
\]

For $z \in (B,B+\eps)$ we have 
\[
|z-r| < B+\eps-r, \quad |z|>B, \quad \text{and} \quad |z-C|>C-B-\eps.
\]
Also note that the restriction of $z^2 +(2/3) z + 1$ to the real line attains its minimum at $-1/3$, where it takes the value $8/9$. In particular, it is positive. Since \[-1/3 < B < B+\eps,\] we have \[|z^2 +(2/3) z + 1| > B^2 +(2/3) B + 1\] if $z \in (B,B+\eps)$. This means that
\[
|f_r(z)| < \frac{B+\eps-r}{B(z-B)(C-B-\eps)(B^2 +(2/3)B + 1)}
\]
on that interval, which leads to the estimate
\begin{align*}
I_r(B,B+\eps) & < \frac{\sqrt{B+\eps-r}}{\sqrt{B(C-B-\eps)(B^2 +(2/3) B + 1)}}\int_B^{B+\eps} \frac{1}{\sqrt{z-B}}dz \\
&= \frac{2\sqrt{\eps(B+\eps-r)}}{\sqrt{B(C-B-\eps)(B^2 +(2/3) B + 1)}}
\end{align*}
upon taking the square root and integrating.

 Similarly, we have
\[
I_r(C-\eps,C) < \frac{2\sqrt{\eps (C-r)}}{\sqrt{(C-\eps)(C-\eps-B)((C-\eps)^2 +(2/3) (C-\eps) + 1)}}.
\]
Putting these estimates together at $r = -9.0795$ yields that \[\delta(-9.0795) > 2.6 \times 10^{-6} > 0.\]
These numerical calculations are contained in the ancillary file \href{https://arxiv.org/src/2404.00336/anc}{\texttt{integrals}}.

We can bound $\delta$ from above at $r = -9.0791$ in a similar way, using that $I_r(B,C)> I_r(B+\eps,C-\eps)$ and \[ I_r(r,0)=I_r(r,r+\eps)+I_r(r+\eps,-\eps)+I_r(-\eps,0)\]
where $\eps = 10^{-12}$. The integrals $I_r(B+\eps,C-\eps)$ and $I_r(r+\eps,-\eps)$ are calculated using interval arithmetic. For $I_r(r,r+\eps)$, we observe that
\[
|f_r(z)| < \frac{z-r}{-(r+\eps)(B-r-\eps)(C-r-\eps)((r+\eps)^2 +(2/3) (r+\eps) + 1)}
\]
as long as $z \in (r,r+\eps)$, where we used that $r+\eps < -1/3$. This yields
\[
I_r(r,r+\eps) < \frac{(2/3) \eps^{3/2}}{\sqrt{-(r+\eps)(B-r-\eps)(C-r-\eps)((r+\eps)^2 +(2/3) (r+\eps) + 1)}},
\]
where the factor $2/3$ comes from the integration of $(z-r)^{1/2}$.

 If  $z \in (-\eps, 0)$, then $|z-B| > B$, $|z-C| > C$, $|z-r|<|r|$ and $|z^2+(2/3)z+1|>\eps^2 -2 \eps /3 +1$ because $-\eps > -1/3$, so that
\[
|f_r(z)| < \frac{|r|}{|z|BC (\eps^2 -2 \eps /3 +1)} = \frac{|r|}{|z|(\eps^2 -2 \eps /3 +1)}
\]
and hence
\[
 I_r(-\eps,0)<\frac{2\sqrt{\eps |r|}}{\sqrt{\eps^2 -2 \eps /3 +1}}.
\]

These lead to the bound
\[
\delta(-9.0791) < -2.36 \times 10^{-6} < 0.
\]

By the intermediate value theorem, $\delta$ has a zero in $(-9.0795,-9.0791)$ and by \lemref{lem:transverse_measure}, that zero is equal to $r_0$.
\end{proof}

In particular, the above lemma proves that $r_0 < 0$, so we can use \propref{prop:formulas} to bound the extremal length of $\beta$. By \corref{cor:reduction_to_6}, this yields bounds on the extremal length of an edge curve on the cube punctured at its vertices.

\begin{cor} \label{cor:upper_bound}
We have \[\el(\beta, \CC \setminus \{0, \rho^2, 1/\rho^2, \omega^2, \wbar{\omega}^2\})\in (1.421585, 1.4215931),\] hence the extremal length of any edge curve on the cube punctured at its vertices is contained in the interval $(2.84317,2.843187)$.
\end{cor}
\begin{proof}
By evaluating the formula in the infimum in \propref{prop:formulas} at \[r=-9.0792887\] (a better estimate for $r_0$ found numerically), we obtain an upper bound for the extremal length of $\beta$. The formula involves three indefinite integrals, which we estimate using similar techniques as in the proof of \lemref{lem:r_0}. These calculations are done in the ancillary file \href{https://arxiv.org/src/2404.00336/anc}{\texttt{integrals}} and yield the bound
\[
E:=\el(\beta, \CC \setminus \{0, \rho^2, 1/\rho^2, \omega^2, \wbar{\omega}^2\}) \leq 1.42159306377875 < 1.4215931.
\]

For the lower bound, we need to estimate 
\[
J_r(C,r) = I_r(C,+\infty) + I_r(-\infty, r)
\]
from above, which is slightly different from previous estimates because the intervals are infinite. However, we can make the change of variable $z=1/y$ to get
\begin{align*}
J_r(C,r) &= \int_{1/r}^{1/C} \sqrt{\frac{|1-ry|}{|1-By||1-Cy||1+(2/3)y +y^2|}}dy \\
& = \sqrt{|r|}\int_{1/r}^{B} \sqrt{\frac{|y-1/r|}{|y-B||y-C||y^2+(2/3)y + 1|}}dy
\end{align*}
(because $BC=1$) and treat this integral like the others. At $r=-9.0792887$, this yields the lower bound
\begin{align*}
E & \geq \frac{4I_r(0,B)}{I_r(r,0) + I_r(B,C) + \frac{J_r(C,r)}{I_r(0,B)}| I_r(r,0) - I_r(B,C)|} \\ &\geq 1.42158510064687 > 1.421585.
\end{align*}

Since the extremal length of an edge curve on the cube is equal to $2E$ by \corref{cor:reduction_to_6}, we can multiply the previous bounds on $E$ by $2$ to estimate its value. This shows that $2E \in (2.84317,2.843187)$.
\end{proof}

Having found good estimates for the extremal length of the edge curves, we next need to show that all other essential simple closed curves on the punctured cube have a larger extremal length.

\section{Geodesic trajectories on the cube} \label{sec:trajectories}

The goal of this section is to prove lower bounds on the extremal length of most essential simple closed curves on $X$, the unit cube punctured at its vertices. We do this by applying the definition of extremal length with the Euclidean metric $h_{\mathrm{flat}}$ on the surface of $X$. Given a homotopy class $c$ of essential simple closed curves, we have
\[
\el(c,X) \geq \frac{\ell(c,{h_{\mathrm{flat}}})^2}{\area(h_{\mathrm{flat}})} = \frac{\ell(c,{h_{\mathrm{flat}}})^2}{6}.
\]
In order to bound $\ell(c,{h_{\mathrm{flat}}})$ from below, a useful fact is that we can pull curves in $c$ tight and then slide them parallel to themselves until they hit a vertex. The resulting curve is on the cube $\overline{X}$ but it is a limit of curves in $X$ that belong to $c$, as shown in \cite[Proposition 2.7]{BolzaEL}.

\begin{lem} \label{lem:geod_rep}
Let $Y$ be a closed Euclidean cone surface $\overline{Y}$ punctured along some finite set $P$. For any homotopy class $c$ of essential simple closed curves in $Y$, there exists a closed curve $\gamma^*$ in $\overline{Y}$ such that $\gamma^* \setminus P$ is geodesic, $\gamma^*$ is the limit of a sequence of simple closed curves in $c$, $\length(\gamma^*,h)=\ell(c,h)$ with respect to the given Euclidean cone metric $h$, and $\gamma^*$ passes through either a point in $P$ or a cone point. 
\end{lem}

Another useful ingredient is the fact, proved in \cite{Davis}, that there is no geodesic trajectory from a vertex to itself on the cube (avoiding all other vertices). This means that the minimal length representative $\gamma^*$ of a homotopy class $c$ provided by \lemref{lem:geod_rep} has to pass through at least two vertices.

This leads us to study geodesic trajectories between distinct vertices on the cube. In trying to classify such trajectories, we rediscovered the follo\-wing phenomenon, first observed in \cite{Troubetzkoy}, which reproves the above cited fact from \cite{Davis} and generalizes easily to the other Platonic solids apart from the dodecahedron (where the statement is false \cite{AthreyaAulicino}).

\begin{thm} \label{thm:rotation}
For any vertex-to-vertex geodesic trajectory on the surface of the regular tetrahedron, cube, octahedron, or icosahedron, there is an order two rotation of the polyhedron that sends the trajectory to itself. In particular, there does not exist a geodesic trajectory from a vertex to itself on these Platonic solids.
 \end{thm}
 
 \begin{proof}
Let $P$ be one of the four Platonic solids in the statement and let $\gamma$ be a vertex-to-vertex geodesic trajectory on $P$. Denote the interior of $\gamma$ by $\gamma^\circ$, which is disjoint from the vertices of $P$ by hypothesis. Consider the regular tiling of $\RR^2$ by tiles congruent to the faces of $P$ (either squares or equilateral triangles) and let $\Lambda$ be the set of vertices of this tiling. By applying a translation if necessary, we may assume that $\Lambda$ contains the origin, in which case it forms a lattice.
 
Observe that $P$ minus its set of vertices and $\RR^2 \setminus \Lambda$ have the same universal cover $U$ (an infinite tree of squares or triangles with their vertices removed). Thus, we can lift $\gamma^\circ$ to $U$ then project it down to some proper geodesic arc $\alpha$ on $\RR^2 \setminus \Lambda$.
 
There is an order $2$ rotation $f$ of $\RR^2 \setminus \Lambda$ that sends $\alpha$ to itself exchanging its endpoints $p$ and $q$, namely, the map $f(x)=p+q-x$. Since $f$ is a isometry and preserves $\Lambda$, it sends tiles to tiles. In particular, it permutes the tiles containing its fixed point $\overline{x}=(p+q)/2$. If $\overline{x}$ lies in the interior of a tile $T$, then $f$ restricts to a rotation of $T$, so that $\overline{x}$ is equal to the center of $T$. Since $f$ has order $2$, this case can only happen if $T$ is a square, i.e., if $P$ is the cube. If $\overline{x}$ lies in the interior of an edge $E$, then $f$ acts as an order-two rotation on this edge, so it fixes its midpoint. The last possibility is that $\overline{x}$
 is a vertex, but this is ruled out since $\overline{x} \in \alpha\subset \RR^2 \setminus \Lambda$.

This shows that the midpoint $m$ of $\gamma$ on $P$ is the midpoint of an edge or the center of a face (only if $P$ is the cube). Note that $P$ admits a rotation $F$ of order $2$ around $m$ by virtue of being Platonic. Moreover, $F$ sends $\gamma$ to itself because it does so locally around the midpoint.

Suppose that the two endpoints of $\gamma$ coincide. Then $F$ has this endpoint as a fixed point. This is a contradiction since the other fixed point of $F$ besides $m$ is diametrically opposite to $m$, and the point diametrically opposite to the center of an edge is the center of an edge and the point diametrically opposite to the center of a square face is the center of a square face. 
 \end{proof}
 
 We use the proof of the above theorem to deduce the following length estimates for geodesic trajectories on the surface of the unit cube.
 
 \begin{lem} \label{lem:adjacent}
 The shortest geodesic trajectory between two adjacent vertices on the unit cube is the edge connecting them (of length $1$), and the second shortest has length $\sqrt{13}$.
 \end{lem}
 \begin{proof}
Let $\gamma$ be a geodesic trajectory between adjacent vertices on the cube, let $\alpha$ be its unfolding to the plane starting at the origin as in the proof of \thmref{thm:rotation}, and let $(x,y)=\alpha(1) \in \ZZ^2$ be its other endpoint. Let $F$ be the order-two rotation of the cube that swaps the endpoints of $\gamma$. Since $F$ sends edges to edges, it fixes the midpoint $M$ of the edge between these endpoints, and hence the diametrically opposite point as well. Recall that the midpoint of $\gamma$ is a fixed point of $F$, from which we deduce that the midpoint of $\alpha$ is not the center of a tile, hence $x$ and $y$ are not both odd. They are not both even either, otherwise $\alpha$ would pass through a lattice point in its interior.  

For the same reason, if $\min(|x|,|y|)=0$, then $\max(|x|,|y|)=1$. This corresponds to the case where $\gamma$ is an edge. Otherwise, $\min(|x|,|y|)\geq1$ and $\max(|x|,|y|)\geq2$ because of the parity condition. If equality holds, then $\alpha$ folds back onto the cube to a path joining two diametrically opposite points, so this possibility is also ruled out. Hence either $\min(|x|,|y|)=1$ and $\max(|x|,|y|)\geq4$, leading to the bound
\[
\length(\gamma)=\length(\alpha)=\sqrt{x^2+y^2} \geq \sqrt{17},
\]
or $\min(|x|,|y|)\geq 2$ and $\max(|x|,|y|)\geq 3$, leading to 
$\length(\gamma)\geq \sqrt{13}$. This is achieved if $(x,y)=(2,3)$, which folds to a trajectory between two adjacent vertices on the cube (see \figref{fig:3_2}).
\end{proof}

 \begin{figure}
    \centering
    \subcaptionbox{One of the two second shortest trajectories between adjacent vertices\label{fig:3_2}}{\includegraphics[width=0.3\textwidth]{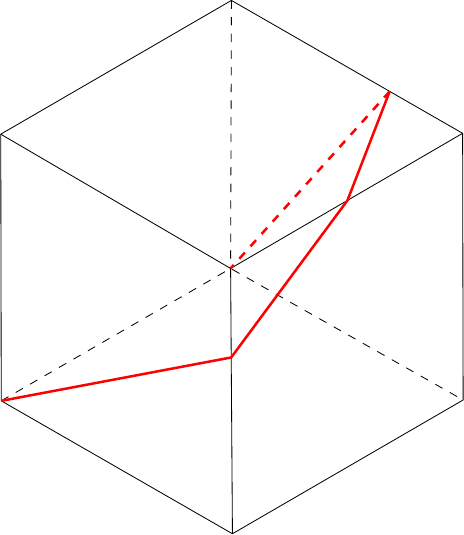}}
    \hfill
    \subcaptionbox{One of the two second shortest trajectories between opposite vertices on a face\label{fig:3_1}}{\includegraphics[width=0.3\textwidth]{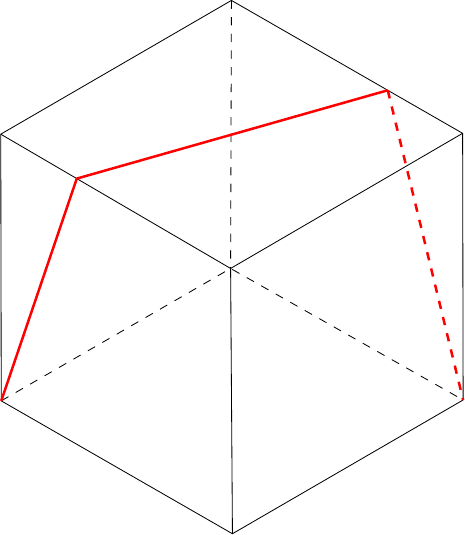}}
    \hfill
    \subcaptionbox{One of the six shortest trajectories between diametrically opposite vertices\label{fig:2_1}}{\includegraphics[width=0.3\textwidth]{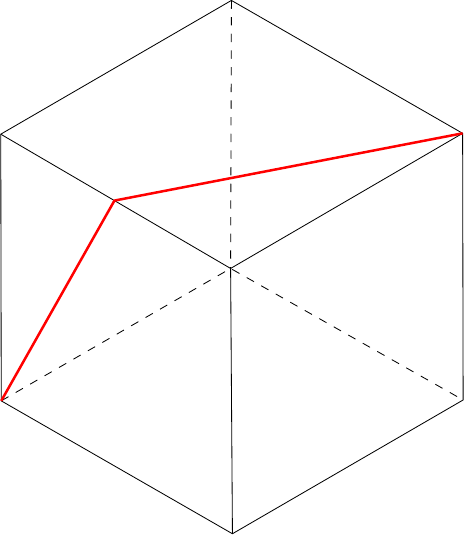}}
    \caption{Various geodesic trajectories between pairs of vertices on the cube.}
\end{figure}
 
  \begin{lem} \label{lem:diagonal}
 The shortest geodesic trajectory between two opposite vertices on a face of the unit cube is the diagonal of length $\sqrt{2}$ and the second shortest has length $\sqrt{10}$.
 \end{lem}
  \begin{proof}
The argument is similar to the previous one except that (using the same notation as before) the fixed points of $F$ are the centers of two opposite faces this time. This implies that $x$ and $y$ are both odd. If \[\min(|x|,|y|)=\max(|x|,|y|)=1,\] then $\gamma$ is a face diagonal of length $\sqrt{2}$. Otherwise, $\min(|x|,|y|)\geq 1$ and $\max(|x|,|y|)\geq 3$ so that $\length(\gamma)\geq \sqrt{10}$. This is achieved for the arc $\alpha$ ending at $(x,y)=(3,1)$ for example, which is indeed the unfolding of a trajectory joining two opposite vertices on a face of the cube (see \figref{fig:3_1}).
 \end{proof}

 \begin{lem} \label{lem:diameter}
 The shortest geodesic trajectory between two diametrically opposite vertices on the unit cube has length $\sqrt{5}$.
 \end{lem}
\begin{proof}
Let the notation be as before. Since the rotations of order $2$ around the centers of opposite faces do not send pairs of diametrically opposite vertices to themselves, $F$ must be a rotation around the midpoints of two opposite edges. We deduce that $x$ and $y$ have opposite parity, so that one of $x/2$ and $y/2$ is an integer and the other is not. Furthermore, if one of them is zero, then the other must be $\pm1$ since the trajectory is not allowed to hit vertices in its interior. As the trajectory $\gamma$ is not an edge, this possibility is ruled out. We thus have $\min(|x|,|y|)\geq 1$ and $\max(|x|,|y|)\geq2$ so that $\length(\gamma) \geq \sqrt{5}$. Equality is achieved for $(x,y)=(2,1)$, which is indeed the unfolding of a trajectory between two diametrically opposite vertices on the cube (see \figref{fig:2_1}).
\end{proof}

Recall that the \emph{simple length spectrum} of a Riemannian surface $(Y,h)$ is defined as
\[
\left\{ \ell(c,h) : c \text{ is the homotopy class of an essential simple closed curve in }Y \right\}.
\]
We will now combine the above estimates together with \lemref{lem:geod_rep} in order to compute the bottom of the simple length spectrum of $X$ with respect to the flat metric of edge length $1$.

We first need to define two more types of curves on $X$. A \emph{triangle curve} is a curve that surrounds an isosceles right triangle obtained by dividing a face of the cube in two along a diagonal (see \figref{fig:triangle}). A \emph{double-edge} curve is the curve depicted in \figref{fig:double-edge} or one of its images by an isometry. 

\begin{prop} \label{prop:simple_length_spectrum}
The four smallest entries in the simple length spectrum of the unit cube punctured at its vertices are $2$, $2\sqrt{2}$, $2 + \sqrt{2}$, $4$, and $3\sqrt{2}$. These are realized only by the edge curves, the diagonal curves, the triangle curves, the face curves or double-edge curves, and the hexagon curves, respectively.
\end{prop}

\begin{proof}
Let $\gamma$ be an essential simple closed curve on $X$. Pull $\gamma$ tight to a length-minimizing curve $\gamma^*$ that passes through at least one vertex on the cube $\overline{X}$ as described in \lemref{lem:geod_rep}.
Then $\gamma^*$ is a concatenation of a certain number of vertex-to-vertex geodesic trajectories on the cube. Furthermore, by \thmref{thm:rotation}, $\gamma^*$ contains at least two such geodesic segments.

First suppose that $\gamma^*$ is made of exactly two geodesic segments (necessarily between distinct vertices by \thmref{thm:rotation}). Then the two vertices are either adjacent, opposite on a face, or diametrically opposite. 

Suppose that the two vertices are adjacent. Since $\gamma^*$ is made of two geodesic segments connecting these two vertices, we have that $\length(\gamma^*) \geq 2$ by \lemref{lem:adjacent}. If equality holds, then $\gamma^*$ is an edge of the cube traced twice. By \lemref{lem:geod_rep}, there is a sequence of simple closed curves $\gamma_n$ homotopic to $\gamma$ that converge to $\gamma^*$. Note that $\gamma^*$ has a neighborhood $U$ whose intersection with $X$ is a pair of pants (two of whose ends are punctures). By an Euler characteristic calculation, every simple closed curve in a pair of pants is peripheral, so once $n$ is large enough so that $\gamma_n \subset U\cap X$, we know that $\gamma_n$ (and hence $\gamma$) is homotopic to an edge curve since the other two peripheral curves in $U\cap X$ are inessential. If $\gamma^*$ is not an edge traced twice, then again by \lemref{lem:adjacent}, it has length at least $1+\sqrt{13} > 3\sqrt{2}$ (both geodesic segments in $\gamma^*$ have length at least $1$ and at least one of them has length at least $\sqrt{13}$).

Suppose now that the two vertices are opposite on a face. By a similar argument as above but using \lemref{lem:diagonal}, we get that $\length(\gamma^*)\geq 2\sqrt{2}$ with equality only if $\gamma$ is homotopic to a diagonal curve. Otherwise we get $\length(\gamma^*)\geq \sqrt{2} + \sqrt{10} > 3\sqrt{2}$.

Lastly, suppose that the two vertices are diametrically opposite. Then $\length(\gamma^*)\geq 2\sqrt{5}>3\sqrt{2}$ by \lemref{lem:diameter}.

\begin{figure}
    \centering
    \includegraphics[width=0.3\textwidth]{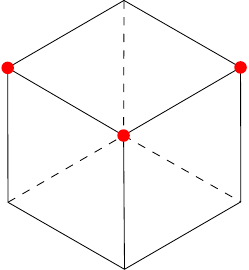}
    \hfill
    \includegraphics[width=0.3\textwidth]{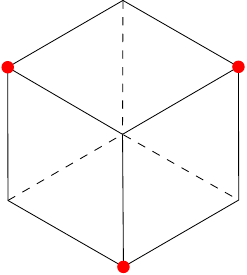}
    \hfill
    \includegraphics[width=0.3\textwidth]{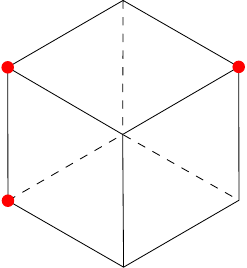}
    \caption{The three triples of distinct vertices on the cube up to isometry.}
    \label{fig:triples}
\end{figure}

Next, suppose that $\gamma^*$ is made of exactly three geodesic segments. Note that the three vertices through which $\gamma^*$ passes are necessarily distinct since the first is distinct from the second, the second from the third, and the third from the first by \thmref{thm:rotation}. Furthermore, is easy to see that up to isometry, there are only three different configurations of three distinct vertices on the cube. Either two pairs of vertices are adjacent and one pair of vertices are opposite on a face, or the three pairs are opposite on faces, or one pair is adjacent, one pair opposite on a face, and one pair diametrically opposite (see \figref{fig:triples}). By the previous lemmata, the length of $\gamma^*$ is at least $2+\sqrt{2}$ in the first case, $3\sqrt{2}$ in the second case, and $1+\sqrt{2}+\sqrt{5}$ in the third case. Now
\[
2+\sqrt{2} < 3\sqrt{2} < 1+\sqrt{2}+\sqrt{5}
\]
so that $\length(\gamma^*) \geq 2+\sqrt{2}$ with equality only if  $\gamma^*$ is an isosceles right triangle forming half a face of the cube. The pair of pants argument above cannot be used in this case since a neighborhood of $\gamma^*$ intersected with $X$ is a sphere with at least $4$ holes or punctures. However, it follows from \lemref{lem:geod_rep} that $\gamma$ can be homotoped to follow $\gamma^*$ most of the time, but circling around the vertices to the left or the right before hitting them. Now, if $\gamma$ was circling any of the vertices from inside the triangle, then it could be homotoped to a much shorter curve, yielding $\ell([\gamma],h_{\mathrm{flat}})<\length(\gamma^*)$, a contradiction. This means that $\gamma$ must go around the vertices of the triangle from the outside, hence that it is a triangle curve. If equality does not hold but the three vertices are in this configuration, then the length of at least one of the three geodesic segments in $\gamma^*$ has to increase to the next smallest trajectory length between its endpoints, yielding \[\length(\gamma^*)\geq \min(1+\sqrt{2}+\sqrt{13},2+\sqrt{10}) = 2+\sqrt{10} > 3\sqrt{2}.\] This means that if $\gamma^*$ passes through exactly three vertices and $\gamma$ is not a triangle curve, then $\length(\gamma^*)\geq 3\sqrt{2}$. If equality holds, then $\gamma^*$ is an equilateral triangle of side length $\sqrt{2}$. In that case, $\gamma$ must pass outside of this triangle at each vertex, otherwise it can be homotoped to a curve of length strictly smaller than $3\sqrt{2}$, because the inner angles at the vertices of the triangle are equal to $\pi/2 < \pi$. Thus $\gamma$ is homotopic to a hexagon curve in this case.

Finally, suppose that $\gamma^*$ is made of at least four geodesic segments. Each geodesic segment has length at least $1$ so that $\length(\gamma^*)\geq 4$. If equality holds, then $\gamma^*$ is a concatenation of four edges. If $\gamma^*$ traces one edge four times, then the pair of pants argument still applies to show that $\gamma$ is an edge curve, but then $\gamma^*$ only traces the edge twice, which is a contradiction. Thus $\gamma^*$ traces out at least two edges. Note that $\gamma^*$ cannot trace exactly three edges, in which case it would not close up. If $\gamma^*$ traces four distinct edges, then it forms the boundary of a face because these are the only embedded cycles of length $4$ on the $1$-skeleton of the cube. In this case, $\gamma$ must be a face curve because if it was circling any vertex from inside the square, then it could be shortened by homotopy to length at most $2+\sqrt{2}+\varepsilon$ for any $\varepsilon>0$. The remaining case is if $\gamma^*$ traces only two distinct adjacent edges, each of them twice. In this case, we can argue that the approximating simple closed curve $\gamma$ must be a double-edge curve. Indeed, $\gamma$ has to circle the two outermost vertices from outside, otherwise we can shorten the length by $1-\varepsilon$. At the inner vertex, the two strands of $\gamma$ have to pass on the side of $\gamma^*$ that crosses an edge (i.e., on the outside of the face containing $\gamma^*$), because if one strand passes on the face side instead, then that strand can be pulled tight to have length $\sqrt{2}$ instead of $2$. Lastly, if $\length(\gamma^*)> 4$ then either there are at least $5$ geodesic segments and $\length(\gamma^*)\geq 5 > 3\sqrt{2}$ or at least one of the edges gets replaced by the next shortest trajectory, so that  $\length(\gamma^*)\geq 3+\sqrt{2} > 3\sqrt{2}$.
\end{proof}

This proposition implies a good lower bound on the extremal length of most essential simple closed curves on the punctured cube $X$.

\begin{cor} \label{cor:lower_flat}
If $\gamma$ is an essential simple closed curve in $X$ different from an edge curve, a diagonal curve, a triangle curve, a face curve, or a double-edge curve, then
\[
\el(\gamma,X) \geq 3,
\]
which is strictly larger than the extremal length of an edge curve.
\end{cor}
\begin{proof}
By \propref{prop:simple_length_spectrum} and the definition of extremal length, we have
\[
\el(\gamma,X) \geq \frac{\ell([\gamma],h_{\mathrm{flat}})^2}{\area(h_{\mathrm{flat}})} \geq \frac{(3\sqrt{2})^2}{6} = 3,
\]
which is strictly larger than the upper bound $2.843187$ on the extremal length of edge curves obtained in \corref{cor:upper_bound}.
\end{proof}

Note that we already know that diagonal curves and face curves have a larger extremal length that the edge curves by \propref{prop:diagonal} and \propref{prop:face}. However, if $\gamma$ is a double-edge curve, then the flat metric gives
\[
\el(\gamma,X) \geq \frac{\ell([\gamma],h_{\mathrm{flat}})^2}{\area(h_{\mathrm{flat}})} \geq \frac{4^2}{6} = \frac{8}{3} = 2 + \frac{2}{3},
\]
which is not good enough to rule it out as realizing the extremal length systole. The lower bound is even smaller for triangle curves. We thus need a different argument for these two types of curves.

\section{Lower bounds from Minsky's inequality and conformal maps} \label{sec:Minsky}

Recall that the \emph{geometric intersection number} $i(\alpha,\beta)$ between two closed curves $\alpha$ and $\beta$ is the smallest  number of intersections (counted with multiplicity) between representatives of their homotopy classes. By the bigon criterion, if $\alpha$ and $\beta$ are transverse and do not form any unpunctured bigon, then $i(\alpha,\beta)$ is simply equal to the number of intersections between them. Minsky's inequality \cite[Lemma 5.1]{MinskyIneq} can be used to obtain lower bounds on extremal length in term of intersection numbers.

\begin{lem}[Minsky's inequality]
Let $Y$ be a Riemann surface with a finitely generated fundamental group and let $\alpha$ and $\beta$ be two essential simple closed curves in $Y$. Then
\[
\el(\alpha,Y)\el(\beta,Y) \geq i(\alpha,\beta)^2.
\]
\end{lem}

This immediately implies the following estimate for the double-edge curves on $X$.

\begin{cor} \label{cor:double-edge}
The extremal length of any double-edge curve on the cube punctured at its vertices is at least $5.627487$.
\end{cor}
\begin{proof}
For any double-edge curve $\alpha$, there is an edge curve $\beta$ such that $i(\alpha,\beta)=4$. For instance, the representatives in Figures \ref{fig:double-edge} and \ref{fig:edge} do not form any bigon and intersect $4$ times. By  Minsky's inequality and \corref{cor:upper_bound}, we obtain
\[
\el(\alpha,X)\geq \frac{i(\alpha,\beta)^2}{\el(\beta,X)} \geq \frac{16}{2.843187} \geq 5.627487
\]
as required.
\end{proof}

Given a triangle curve $\alpha$, we can find a staple curve $\beta$ that intersects it $4$ times. By Remark \ref{rem:staple}, this yields
\[
\el(\alpha,X)\geq \frac{i(\alpha,\beta)^2}{\el(\beta,X)} \geq \frac{16}{6.9282032302755+2.74 \times 10^{-14}},
\]
but the right-hand side is smaller than $2.843188$, so it is not good enough.

We thus replace the staple curve by a \emph{diamond curve} that surrounds the union of an edge and an adjacent face diagonal that lie in a common plane of reflection (see \figref{fig:diamond}). Computing the extremal length of a diamond curve would be as hard as computing the extremal length of an edge curve, which was quite tedious. Instead, we prove a good upper bound by finding a large embedded annulus.

\begin{prop} \label{prop:diamond}
The extremal length of any diamond curve  on the cube punctured at its vertices is at most $5.535645$.
\end{prop}
\begin{proof}
We represent the punctured cube by $Y = \CHAT \setminus \{ \pm a, \pm i a, \pm 1/a, \pm i/a \}$ where $a = (\sqrt{3}+1)/\sqrt{2}$ as in \lemref{lem:plus_sign}. One diamond curve is given by a curve $\beta$ surrounding $[-a,1/a]$.

Consider the annulus $U_1 = \CC \setminus  \bigcup_{n=1}^4 I_n$ where 
\[I_1=(-\infty,-a], I_2 = [-1/a, a], I_3 = i [1/a, \infty), \text{ and } I_4 = i (-\infty, -1/a].\]
We have that $U_1 \subset Y$ and the core curves in $U_1$ are homotopic to $\beta$. As such, we have 
\[
\el(\beta, Y) \leq \el(U_1)
\]
by the the third part of Jenkins's theorem (or by the monotonicity of extremal length under conformal embeddings). We will now compute $\el(U_1)$ explicitly through a sequence of conformal maps.

We first multiply by $-ia$ to obtain the complement $U_2$ of the intervals $i[a^2, \infty)$, $i[-a^2, 1]$, $[1, \infty)$ and $(-\infty, -1]$. Recall that the sine function maps the open vertical strip 
\[
S = \left\{ z \in \CC \mid -\frac{\pi}{2} < \re(z) < \frac{\pi}{2} \right\}
\]
biholomorphically onto the set $ \CC \setminus \left((-\infty, -1] \cup [1,\infty)\right)$. Denote its inverse on that set by $\arcsin$, and let $U_3 = \arcsin(U_2)$.

Since $\sin(iy)= i \sinh(y)$ for all $y\in \RR$, we have that $\arcsin(iv) = i \arcsinh(v)$ for all $v \in \RR$ as well. This means that $\arcsin(i[a^2, \infty)) = i [\arcsinh(a^2),\infty)$ and $\arcsin(i[-a^2, 1])=i [-\arcsinh(a^2),\arcsinh(1)]$, so that
\[
U_3 = S \setminus \left( i [-\arcsinh(a^2),\arcsinh(1)] \cup i [\arcsinh(a^2),\infty) \right)
\]

We then apply $f(z)=e^{2iz}$, which sends $S$ to $\CC\setminus (-\infty, 0]$. To compute the images of the two slits, observe that for every $x\in \RR$ we have 
\[
e^{2\arcsinh(x)} = \left(e^{\arcsinh(x)}\right)^2 = \left(x+\sqrt{x^2+1} \right)^2=:g(x),
\]
so that 
\[
U_4 := f(U_3) = \CC \setminus (-\infty, z_1] \cup [z_2, z_4])
\]
where $z_1 = g(-a^2)$, $z_2 = g(-1)$ and $z_4 = g(a^2)$.

We apply one last transformation $\displaystyle h(z) = \frac{z_4-z}{z_4-z_2}$ to get \[U_5 := h(U_4) = \CC \setminus ( [0, 1] \cup [t, \infty))\] where $\displaystyle t = \frac{z_4-z_1}{z_4-z_2}$. By the conformal invariance of extremal length and \lemref{lem:add_punctures}, we find 
\[
\el(U_1) = \el(U_5) = 2(K/K')(1/\sqrt{t}).
\]

Note that we can write
\begin{align*} 
t &= \frac{z_4-z_1}{z_4-z_2} \\
&= \frac{g(a^2)-g(-a^2)}{g(a^2)-g(-1)} \\
&= \frac{\left(a^2+\sqrt{a^4+1} \right)^2-\left(-a^2+\sqrt{a^4+1} \right)^2}{\left(a^2+\sqrt{a^4+1} \right)^2 -\left(\sqrt{2}-1 \right)^2} \\
&= \frac{4a^2 \sqrt{a^4+1}}{\left(a^2+\sqrt{a^4+1} \right)^2 -\left(\sqrt{2}-1 \right)^2}.
\end{align*}
Using the value of $a = \frac{\sqrt{3}+1}{\sqrt{2}}$, we get $a^2 = 2 +\sqrt{3}$ and
\[
a^4 + 1 =  8+4\sqrt{3} = \left( \sqrt{2}(\sqrt{3}+1) \right)^2
\]
so that $\sqrt{a^4 + 1}=\sqrt{2}(\sqrt{3}+1) = 2a$. The expression for $t$ thus simplifies to
\[t = \frac{8a^3}{(a^2 + 2a)^2 -\left(\sqrt{2}-1 \right)^2}.\]

The \texttt{Arb} package in \texttt{SageMath} (see the ancillary file \href{https://arxiv.org/src/2404.00336/anc}{\texttt{integrals}}) gives
\[ 2(K/K')(1/\sqrt{t}) < 5.53564498781518 < 5.535645. \qedhere\]
\end{proof} 

Together with Minsky's inequality, this yields an improved lower bound on the extremal length of triangle curves.

\begin{cor} \label{cor:triangle}
The extremal length of any triangle curve on the cube punctured at its vertices is at least $2.890358$.
\end{cor}
\begin{proof}
Let $\alpha$ be the triangle curve in \figref{fig:triangle} and $\beta$ the diamond curve in \figref{fig:diamond}. Then $\alpha$ and $\beta$ are transverse, do not form any bigon, and intersect $4$ times, so that $i(\alpha,\beta)=4$. Minky's inequality then yields
\[
\el(\alpha, X) \geq \frac{i(\alpha,\beta)^2}{\el(\beta,X)} \geq \frac{16}{5.535645} > 2.890358. \qedhere
\]
\end{proof}

\section{Finishing the proof} \label{sec:last}

We now combine all the previous extremal length estimates to obtain our main result.

\begin{thm}
The extremal length systole of the cube punctured at its vertices is realized by the edge curves.
\end{thm}
\begin{proof}
Let $\gamma$ be an edge curve on the punctured cube $X$. By \corref{cor:upper_bound}, we have \[\el(\gamma,X) < 2.843187.\] If $\alpha$ is a double-edge curve, then \[\el(\alpha,X)\geq 5.627487 > \el(\gamma,X)\] by \corref{cor:double-edge}, so that $\alpha$ is not a systole. If $\alpha$ is a triangle curve, then \[\el(\alpha,X)\geq 2.890358 > \el(\gamma,X)\] by \corref{cor:triangle}, so $\alpha$ is not a systole. If $\alpha$ is a face curve, then
\[
\el(\alpha,X) \geq 3.12680384539222-8.07\times10^{-15} > \el(\gamma,X)
\]
by \propref{prop:face}, so $\alpha$ is not a systole. If $\alpha$ is a diagonal curve, then
\[
\el(\alpha,X) \geq 4.1335929781133-2.81\times10^{-14} > \el(\gamma,X)
\]
by \propref{prop:diagonal}, so $\alpha$ is not a systole. If $\alpha$ is an essential simple closed curve in $X$ which is not an edge curve, a face curve,  a diagonal curve, a double-edge curve, or a triangle curve, then
\[
\el(\alpha,X) \geq  3 > \el(\gamma,X).
\]
by \corref{cor:lower_flat}, so $\alpha$ is not a systole. Since the extremal length systole can only be realized by essential simple closed curves \cite[Theorem 3.2]{BolzaEL}, it is realized by the edge curves and only them.
\end{proof}

In other words, the extremal length systole of the cube is equal to the extremal length of the edge curves, which is the first half of \thmref{thm:main}. The second half of the theorem is obtained by multiplying the first formula in \propref{prop:formulas} by $2$ in view of \corref{cor:reduction_to_6}. Note that the formula in the introduction is written in terms of $F_r = f_r/3$ instead of $f_r$, which does not change the ratio of integrals.

\bibliographystyle{amsalpha}
\bibliography{biblio}

\end{document}